\documentclass[10pt,reqno]{amsart}
\usepackage{epsfig,amssymb,amsmath,version}
\usepackage{amssymb,version,graphicx,fancybox,mathrsfs}
\usepackage{url,hyperref}
\usepackage{subfigure}
\usepackage{color}
\usepackage{stmaryrd}
\usepackage{multirow}
\usepackage{booktabs,siunitx,enumitem}
\usepackage{booktabs}
\usepackage[ruled,vlined]{algorithm2e}
\usepackage[utf8]{inputenc}
\usepackage[english]{babel}
\usepackage{amsmath, bm}

\catcode`\@=11 \theoremstyle{plain}
\@addtoreset{equation}{section}   

\renewcommand\thefigure{\@arabic\c@figure}
\newtheorem{thm}{\bf Theorem}

\newenvironment{theorem}{\begin{thm}} {\end{thm}}
\newtheorem{cor}{\bf Corollary}
\newtheorem{prop}{Proposition}[section]

\newtheorem{lmm}{\bf Lemma}

\newenvironment{lemma}{\begin{lmm}}{\end{lmm}}
\theoremstyle{remark}
\newtheorem{rem}{\bf Remark}

\def \epsilon {{\varepsilon}}

\definecolor{bgblue}{rgb}{0.04,0.39,0.54}
\definecolor{lired}{rgb}{0.3, 0.0, 0.0}
\definecolor{ligreen}{rgb}{0.0, 0.3, 0.0}
\definecolor{liblue}{rgb}{0.9, 1.0, 1.0}
\definecolor{gray}{rgb}{0.6, 0.6, 0.6}
\definecolor{sky}{rgb}{0.3, 1.0, 1.0}
\definecolor{bunhong}{rgb}{1.0, 0.3, 1.0}
\definecolor{yellow}{rgb}{0.97, 1, 0.0}
\definecolor{liyellow}{rgb}{0.9, 0.8, 0.0}
\definecolor{cengse}{rgb}{0.00,0.40,0.29}

\renewcommand \wedge \times

\begin{document}

{\title[on a new class of BDF and IMEX schemes] {
		On a new class of BDF and IMEX schemes for  parabolic type equations}
	
	\author[
	F. Huang and J. Shen
	]{
		Fukeng Huang$^\dag$ and Jie Shen$^\ddag$
	}
	\thanks{$^\dag$ Department of Mathematics, National	University of Singapore, Singapore, 119076 (hfkeng@nus.edu.sg). \\
	{}\hskip 10pt	$^\ddag$ School of Mathematical Science, Eastern Institute of Technology, Ningbo, 315200, China (jshen@eitech.edu.cn). This work is supported in part by NSFC grant  11971407.}

	
	\begin{abstract}
		{\color{black} When applying the classical multistep schemes for  solving differential equations}, one often faces the dilemma that smaller time steps are needed with higher-order schemes, making it  impractical to use high-order schemes for stiff problems.
		We construct in this paper a new class of BDF and implicit-explicit (IMEX)  schemes  for  parabolic type equations based on the Taylor expansions at time $t^{n+\beta}$ with $\beta > 1$ being a tunable parameter. 	{\color{black}These new schemes, with a suitable $\beta$,  allow  larger time steps at higher-order for stiff problems than that is allowed with a usual higher-order scheme. For parabolic type equations, we identify an explicit uniform  multiplier for the new second- to fourth-order schemes, and  conduct rigorously stability and error analysis   by using the energy argument. We also present  ample numerical examples to validate our findings.}
	\end{abstract}
	\maketitle
	
	\section{Introduction}
	We consider in this paper numerical methods of a class of nonlinear ordinary or partial differential equations in the form
	\begin{equation}\label{nonlinear}
		\begin{split}
			& u_t+\mathcal{L}u(t)+\mathcal{G}[u(t)]=f(t),\quad 0<t<T,\\
			& u(0)=u^0,
		\end{split}
	\end{equation}
	where $\mathcal{L}$ is a linear (or possibly nonlinear) positive  operator  and  $\mathcal{G}$ is a nonlinear operator, whose exact descriptions can be found in the next section.
	
	Numerical approximation of  ordinary  differential equations (ODEs) is a very mature field (see, for instance, \cite{butcher1987numerical,butcher2016numerical,hairer2006geometric,hundsdorfer2003numerical}), and the numerical methods developed for ODEs have been playing important roles in solving partial differential equations (PDEs) in the form of \eqref{nonlinear} through the method of lines \cite{schiesser2012numerical}, or the so called method of lines transpose  \cite{jia2008krylov}, i.e.,  discretizing first in time followed by the discretization in space. In particular, the backward difference formulae (BDF) and the implicit-explicit (IMEX) schemes are frequently used to deal with \eqref{nonlinear} which exhibit stiff behaviors \cite{ascher1995implicit,higham1993stiffness,kassam2005fourth}.
	
	Two key  issues of numerical methods for \eqref{nonlinear} are stability and accuracy. In order to obtain highly accurate solution with less computational costs, it is highly desirable to be able to use higher-order schemes with larger time steps.
	However, as we increase the order of accuracy of BDF or IMEX type schemes, their stability regions usually decrease, i.e., smaller time steps need to be used with higher-order schemes, particularly for stiff problems, making high-order schemes impractical for many complex nonlinear systems. A natural question arises: is it possible to develop higher-order multi-step schemes such that their stability regions are comparable or even larger than lower-order classical BDF or IMEX schemes?
	
	The main purposes of this paper are two-fold:
	\begin{itemize}
		\item to construct a new class  of  BDF and IMEX schemes with a tunable parameter such that larger time steps can be used in higher-order schemes;
		\item to  carry out a rigorous stability and error analysis for  this new class of IMEX schemes.
	\end{itemize}
	Furthermore, we provide convincing numerical evidences to validate our theoretical findings.
	
	We recall that the  classical BDF and IMEX schemes for approximating solution at time $t^{n+1}$   are usually constructed using the Taylor expansion formulae at time $t^{n+\beta}$ with $\beta\in \{0,1\}$. In this paper, we shall construct a new class of  BDF and  IMEX schemes based on the Taylor expansion formulae at time $t^{n+\beta}$ with $\beta\ge 1$ being a tunable parameter. The new schemes are a simple generalization of the classical  BDF or IMEX schemes with essentially the same computational efforts. However, they enjoy a remarkable property that their stability regions increase as the parameter $\beta$ increases, making it possible, by choosing a suitably large $\beta$,  to use high-order schemes with reasonably larger time steps. The price to pay with a larger $\beta$ is  increased truncation errors which can be more than compensated with higher-order of accuracy.
	
	On the other hand, it is well known that a rigorous stability and error analysis {\color{black} by using the energy technique} of the  classical BDF (and the related IMEX)  schemes of order up to five  (cf. \cite{akrivis2017combining,akrivis2015fully,MR3755672,MR4048622,MR3119720}) relies on a result by Nevanlinna and Odeh \cite{nevanlinna1981} (see also \cite{BDF6} for the extension to the six-order BDF scheme) in which the existence of suitable multiplier that can lead to energy stability was established. It is therefore  natural to ask whether such a  multiplier exists for the new
	class of BDF schemes. We shall construct explicitly suitable multipliers in a more general form
	for the new class of BDF schemes of orders two to four, and derive explicit telescoping formulae associated with these multipliers. Furthermore, for nonlinear parabolic type equations, we show rigorously that the stability condition of the new class of IMEX schemes  becomes  less restrictive as $\beta$ increases, particularly compared with the classical case of $\beta=1$.
	
	The idea behind the new class of BDF and IMEX schemes is very simple but original,  and can be easily extended to other type numerical schemes. However, our stability and error analysis rely on the explicit formulae for the uniform multipliers and telescoping decomposition whose derivations are totally nontrivial and original. On the other hand,
	the new schemes can be easily  implemented with a minimal effort by modifying  the  code based on the classical BDF or IMEX schemes, and  provide a much needed improvement on the stability of higher-order schemes.

	The rest of the paper is organized as follows. In Section 2, we describe the abstract setting and construct  the  new class of BDF and IMEX  methods  based on the Taylor expansion at time $t^{n+\beta}$ and investigate their stability regions. In Section 3, we identify an explicit and uniform multiplier for the new class of BDF and  IMEX  schemes, which plays an essential role in the stability and error analysis. In Section 4, we establish the unconditional stability for the linear parabolic equations and the stability, followed by error analysis for the nonlinear parabolic equations in Section 5. In section 6, we discuss extension  to the fifth-order scheme. In section 7, we provide  numerical examples to show the advantages of our new schemes, followed by some concluding remarks  in section 8.

\section{A new class of BDF and IMEX  schemes}\label{IMEXformula}
\subsection{The abstract setting}
We first describe the functional setting. For the sake of simplicity, we consider a simpler setting than that used in \cite{akrivis2015fully}, although our analysis would also work for the more general setting there.

Let $V$ and $H$ be two real Hilbert spaces such that $V \subset H=H' \subset V'$, with $V$ densely and continuously embedded in $H$ and $V'$ being the dual space of $V$. We consider \eqref{nonlinear} with
  $\mathcal{L}$: $V \rightarrow V'$ being a positive definite, {\color{black}self-adjoint}, linear operator,   
   and $f$ in $V'$ is a given source term. We denote the inner product in $H$ by $(\cdot, \cdot)$, and the induced norm in $H$ by $|\cdot|$. We also denote the norm  in $V$  by $\|\cdot\|$ which is defined as  $\|u\|:=|\mathcal{L}^{1/2}u|=(\mathcal{L}u,u)^{1/2}$. The dual norm in $V'$ is defined by
\begin{equation}\label{starnorm}
\|v\|_{\star}:=\mathop{\rm{sup}}\limits_{u \in V \backslash \{0\}} \frac{|(v,u)|}{\|u\|},\quad \forall v \in V'.
\end{equation}
We assume that  the nonlinear operator $\mathcal{G}$ satisfies the following local Lipschitz condition  \cite{akrivis2015fully} in a ball {\color{black} $\mathcal{B}_{u(t)}:=\{v \in V: \|v-u(t)\| \le 1\}$}, centered at the exact solution  $u(t)$,
\begin{equation}\label{LocalLp}
\|\mathcal{G}(v)-\mathcal{G}(\tilde{v})\|^2_{\star} \le \gamma\|v-\tilde{v}\|^2+\mu |v-\tilde{v}|^2, \quad \forall v, \tilde{v} \in \mathcal{B}_{u(t)},\quad \forall t \in [0,T],
\end{equation}
 with a non-negative constant $\gamma$ and an arbitrary constant $\mu$.



\subsection{Construction of the new  schemes}
We shall first construct the new schemes for \eqref{nonlinear}  based on the Taylor expansion at time $t^{n+\beta}$. Given an integer $k\ge 2$, denoting $t^n= n\Delta t$, it follows from the Taylor expansion at time $t^{n+\beta}$ that
\begin{equation}\label{Taylor}
\phi(t^{n+1-i})=\sum_{m=0}^{k-1}[(1-i-\beta)\Delta t]^m \frac{\phi^{(m)}(t^{n+\beta})}{m!}+\mathcal{O}(\Delta t^{k}),\quad \text{ for } k\ge  i \ge 0.
\end{equation}
Then we can derive from the above an implicit difference formula to approximate $\partial_t\phi(t^{n+\beta})$:
\begin{equation}\label{TaylorA}
\frac{1}{\Delta t}\sum_{q=0}^{k}a_{k,q}(\beta)\phi(t^{n+1-k+q})=\partial_t\phi(t^{n+\beta})+\mathcal{O}(\Delta t^k),
\end{equation}
where  $a_{k,q}(\beta)$ can be uniquely determined by solving the following linear system with a Vandermonde matrix:
\begin{equation}\label{solve_akq}
{\left[
\begin{array}{ccccc}
1&1&...&...&1\\
\beta-1&\beta&...&...&\beta+k-1\\
(\beta-1)^2&\beta^2&...&...&(\beta+k-1)^2\\
\vdots&\vdots&\vdots&\vdots&\vdots\\
(\beta-1)^k&\beta^k&...&...&(\beta+k-1)^k
\end{array}
\right]}{\left[
\begin{array}{c}
a_{k,k}(\beta)\\a_{k,k-1}(\beta)\\a_{k,k-2}(\beta)\\ \vdots \\a_{k,0}(\beta)
\end{array}
\right]}=
{\left[
\begin{array}{c}
0\\-1\\0\\ \vdots \\0
\end{array}
\right]}.
\end{equation}
Similarly, we can derive an implicit difference formula to approximate $\phi(t^{n+\beta})$:
\begin{equation}\label{TaylorB}
\sum_{q=0}^{k-1}b_{k,q}(\beta)\phi(t^{n+2-k+q})=\phi(t^{n+\beta})+\mathcal{O}(\Delta t^k),
\end{equation}
with $b_{k,q}(\beta)$ being the unique solution of the following Vandermonde system:
\begin{equation}\label{solve_bkq}
{\left[
\begin{array}{ccccc}
1&1&...&...&1\\
\beta-1&\beta&...&...&\beta+k-2\\
\vdots &\vdots&\vdots&\vdots&\vdots\\
(\beta-1)^{k-1}&\beta^{k-1}&...&...&(\beta+k-2)^{k-1}
\end{array}
\right]}{\left[
\begin{array}{c}
b_{k,k-1}(\beta)\\b_{k,k-2}(\beta)\\\vdots\\b_{k,0}(\beta)
\end{array}
\right]}=
{\left[
\begin{array}{c}
1\\0\\\vdots\\0
\end{array}
\right]}.
\end{equation}
To deal with the nonlinear term in \eqref{nonlinear}, we also need the following explicit difference formula to approximate $\phi(t^{n+\beta})$:
\begin{equation}\label{TaylorC}
	\sum_{q=0}^{k-1}c_{k,q}(\beta)\phi(t^{n+1-k+q})=\phi(t^{n+\beta})+\mathcal{O}(\Delta t^k),
\end{equation}
where $c_{k,q}(\beta)$ can be uniquely determined from:
\begin{equation}\label{solve_ckq}
	{\left[
		\begin{array}{ccccc}
			1&1&...&...&1\\
			\beta&\beta+1&...&...&\beta+k-1\\
			\vdots &\vdots&\vdots&\vdots&\vdots\\
			\beta^{k-1}&(\beta+1)^{k-1}&...&...&(\beta+k-1)^{k-1}
		\end{array}
		\right]}{\left[
		\begin{array}{c}
			c_{k,k-1}(\beta)\\c_{k,k-2}(\beta)\\\vdots\\c_{k,0}(\beta)
		\end{array}
		\right]}=
	{\left[
		\begin{array}{c}
			1\\0\\\vdots\\0
		\end{array}
		\right]}.
\end{equation}
Then, a new class of  BDF schemes for \eqref{nonlinear}  with $\mathcal{G}=0$ is
\begin{equation}\label{genBDF}
	\frac{1}{\Delta t}\sum_{q=0}^{k}a_{k,q}(\beta)\phi^{n+1-k+q}+\mathcal{L}(\sum_{q=0}^{k-1}b_{k,q}(\beta)\phi^{n+2-k+q})=f(t^{n+\beta}),\quad k\ge 2,
\end{equation}
and a new class of  IMEX schemes for \eqref{nonlinear} is
\begin{equation}\label{genIMEX}
	\frac{1}{\Delta t}\sum_{q=0}^{k}a_{k,q}(\beta)\phi^{n+1-k+q}+\mathcal{L}(\sum_{q=0}^{k-1}b_{k,q}(\beta)\phi^{n+2-k+q})+\mathcal{G}(\sum_{q=0}^{k-1}c_{k,q}(\beta)\phi^{n+1-k+q})=f(t^{n+\beta}),\quad k\ge 2.
\end{equation}
\begin{rem}
	When $\beta=1$, \eqref{genIMEX} (resp. \eqref{genBDF}) becomes the classical semi-implicit IMEX (resp. BDF) schemes, and there have been extensive works regarding its stability and error analysis \cite{akrivis2015,akrivis2016backward,akrivis2015fully,MR4048622,lubich1990} in the literature. {\color{black} For $\forall \beta>1$,  \eqref{genBDF} and \eqref{genIMEX}  still involve values at  the same $k+1$-levels  as the classical  one (with $\beta=1$) on the left hand side while they involve values at time $t^{n+\beta}$ on the right hand side.}
\end{rem}

For the reader's convenience, we list below the coefficients in  \eqref{genIMEX}  for $k=2,3,4$.

\noindent $k=2$:
\begin{small}
		\begin{subequations}\label{IMEX2coeff}
		\begin{align}
		&a_{2,2}(\beta)=\frac{2\beta+1}{2},\quad a_{2,1}(\beta)=-2\beta,\quad a_{2,0}(\beta)=\frac{2\beta-1}{2},\\
		&b_{2,1}(\beta)=\beta,\quad b_{2,0}(\beta)=-(\beta-1),\\
	&	c_{2,1}(\beta)=\beta+1,\,c_{2,0}(\beta)=-\beta.
	\end{align}
		\end{subequations}
		\noindent $k=3$:
	\begin{subequations}\label{IMEX3coeff}
		\begin{align}
			&a_{3,3}(\beta)=\frac{3\beta^2+6\beta+2}{6},\,a_{3,2}(\beta)=\frac{-(9\beta^2+12\beta-3)}{6},\,a_{3,1}(\beta)=\frac{9\beta^2+6\beta-6}{6},\,a_{3,0}(\beta)=\frac{-(3\beta^2-1)}{6},\\
			&b_{3,2}(\beta)=\frac{\beta^2+\beta}{2},\,b_{3,1}(\beta)=-(\beta^2-1),\,b_{3,0}(\beta)=\frac{\beta^2-\beta}{2},\\
		&	c_{3,2}(\beta)=\frac{\beta^2+3\beta+2}{2},\,c_{3,1}(\beta)=-(\beta^2+2\beta),\,   c_{3,0}(\beta)=\frac{\beta^2+\beta}{2}.
		\end{align}
	\end{subequations}
	\noindent $k=4$:
	\begin{subequations}\label{IMEX4coeff}
		\begin{align}
			&a_{4,4}(\beta)=\frac{2\beta^3+9\beta^2+11\beta+3}{12},\,a_{4,3}(\beta)=\frac{-8\beta^3-30\beta^2-20\beta+10}{12},\,a_{4,2}(\beta)=\frac{12\beta^3+36\beta^2+6\beta-18}{12}\nonumber\\
			&a_{4,1}(\beta)=\frac{-8\beta^3-18\beta^2+4\beta+6}{12},\,a_{4,0}(\beta)=\frac{2\beta^3+3\beta^2-\beta-1}{12},\\
			&b_{4,3}(\beta)=\frac{\beta^3+3\beta^2+2\beta}{6},\,b_{4,2}(\beta)=\frac{-\beta^3-2\beta^2+\beta+2}{2},\,b_{4,1}(\beta)=\frac{\beta^3+\beta^2-2\beta}{2},\,b_{4,0}(\beta)=\frac{-\beta^3+\beta}{6},\\
	&c_{4,3}(\beta)=\frac{\beta^3+6\beta^2+11\beta+6}{6},\,c_{4,2}(\beta)=\frac{-\beta^3-5\beta^2-6\beta}{2},\,c_{4,1}(\beta)=\frac{\beta^3+4\beta^2+3\beta}{2},\nonumber\\
			& c_{4,0}(\beta)=\frac{-\beta^3-3\beta^2-2\beta}{6}.
		\end{align}
	\end{subequations}
\end{small}
{\color{black}
	\begin{rem}
		Instead of deriving \eqref{genIMEX} from Taylor expansions, one may also derive it  by following the standard construction of the usual multistep methods using interpolation formulae (see, e.g., Section 2 in \cite{iserles2009first}). In fact, it can be shown that  the coefficients $a_{k,q}(\beta), b_{k,q}(\beta),c_{k,q}(\beta) $ can be determined by the values at $t^{n+\beta}$ of the corresponding Lagrange polynomials and their derivatives. For example,
		\begin{equation}
			a_{k,q}(\beta)=\Delta t L_q'(t^{n+\beta}),\quad q=0,...,k,
		\end{equation}
		where $L_q$  is the Lagrange polynomials associated with $t^{n+1-k},...,t^{n+1}$.
	\end{rem}
}
\subsection{Linear stability regions} In this subsection, we  investigate the regions of linear stability of the new schemes \eqref{genBDF}.
For the test equation $\phi_t= \lambda \phi$,  \eqref{genBDF} reduces to
\begin{equation}\label{genIMEX34}
	\frac{1}{\Delta t}\sum_{q=0}^{k}a_{k,q}(\beta)\phi^{n+1-k+q}=\lambda\sum_{q=0}^{k-1}b_{k,q}(\beta)\phi^{n+2-k+q},\quad k\ge 2.
\end{equation}
In order to study the stability regions for  $\beta \ne 1$, we set $\phi^n=w^n$ (here, $``n"$ is an upper index in $\phi^n$ and an exponent in $w^n$) and $z=\lambda \Delta t$ in \eqref{genIMEX34} to obtain its characteristic equation, e.g., in the case of $k=2$, it takes the form:
\begin{equation}\label{eq:chara}
	(2\beta+1-2\beta z)w^2+(2(\beta-1)z-4\beta)w+(2\beta-1)=0.
\end{equation}
Then the region of  absolute stability of method \eqref{genIMEX34} is the set of all $z \in \mathbb{C}$  such that  the characteristic polynomial satisfies the root condition.
 We recall that the second order case  was already considered in \cite{huang2023stability}, and it was shown that   the  second-order case of \eqref{genIMEX34} is  A-stable for $\beta \ge 1$, and more importantly, the  stability regions increase as we increase $\beta$.

 In Fig. \ref{fig:stability3} and Fig. \ref{fig:stability4} , we plot  the stability regions of the general  third- and fourth-order BDF schemes  for $\beta=1,3,5$. We observe that, the  stability regions increase as we increase $\beta$.
\begin{figure}[!htbp]
 \centering
  \subfigure[third order, $\beta=1$]{ \includegraphics[scale=.22]{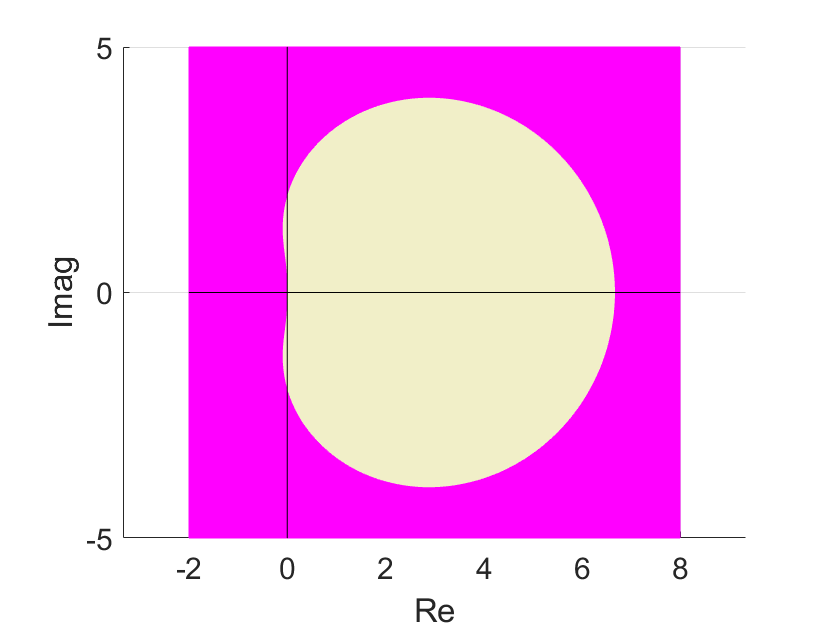}}
  \subfigure[third order, $\beta=3$]{ \includegraphics[scale=.22]{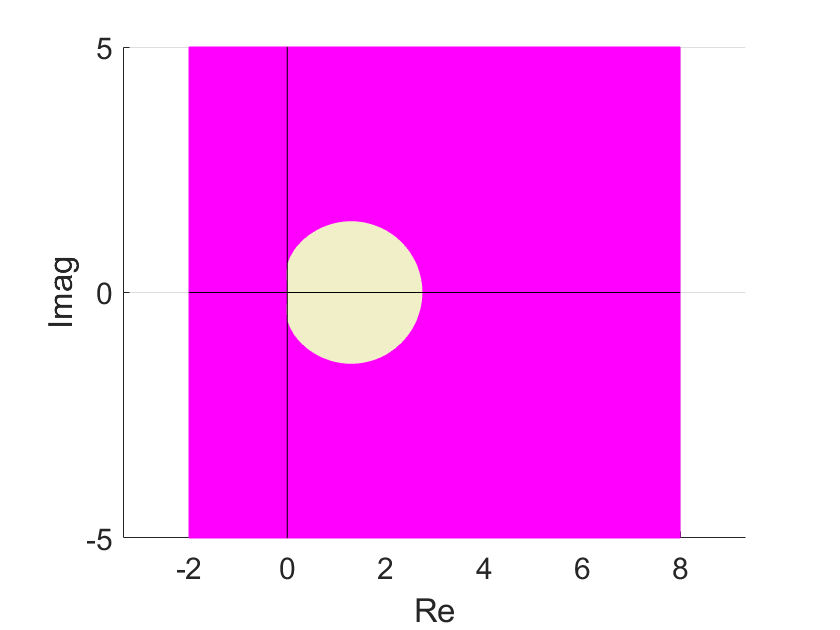}}
  \subfigure[third order, $\beta=5$]{ \includegraphics[scale=.22]{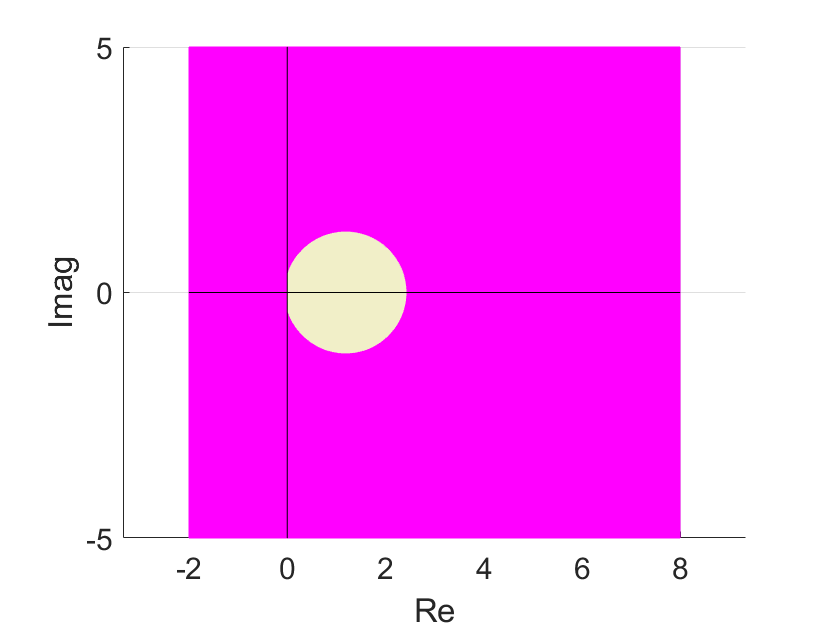}}\\
  \subfigure[$\beta=1$, zoom in around the origin]{ \includegraphics[scale=.22]{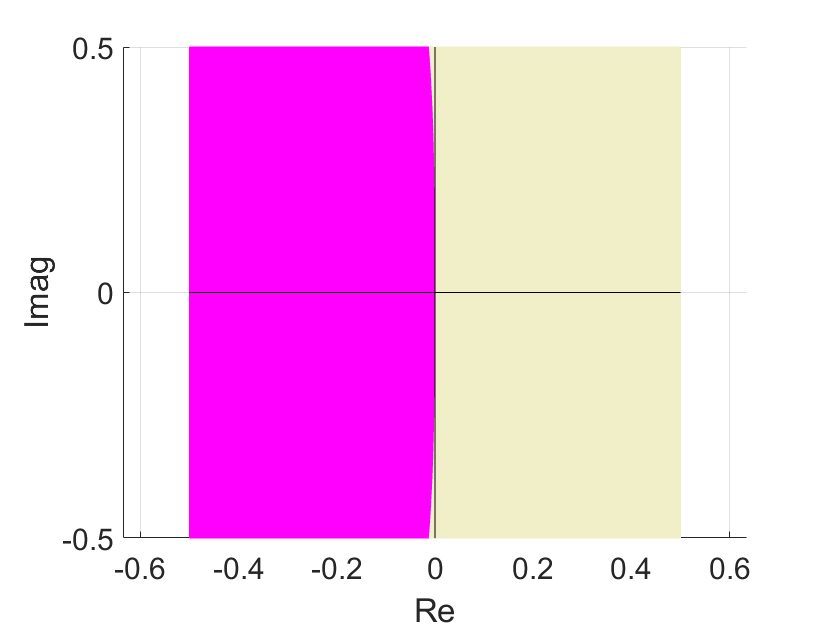}}
  \subfigure[$\beta=3$, zoom in around the origin]{ \includegraphics[scale=.22]{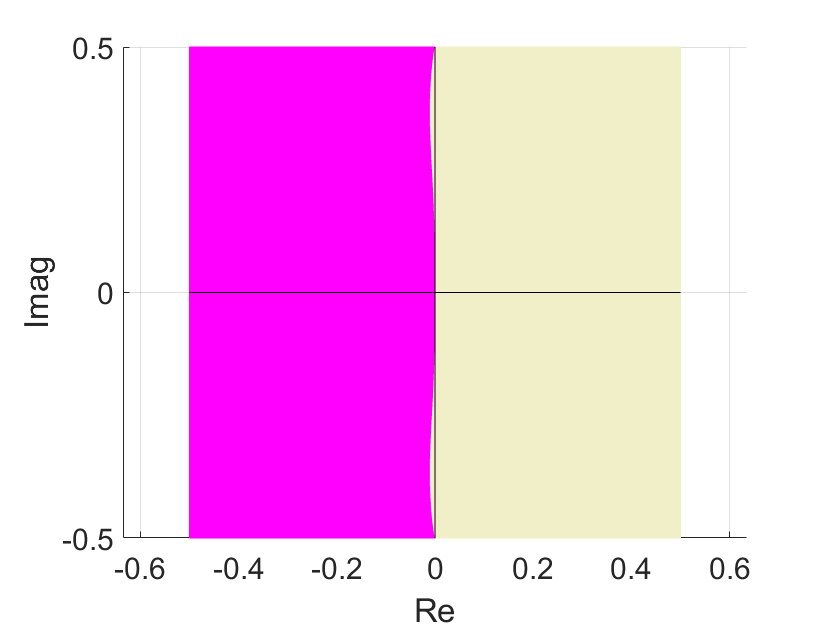}}
  \subfigure[$\beta=5$, zoom in around the origin]{ \includegraphics[scale=.22]{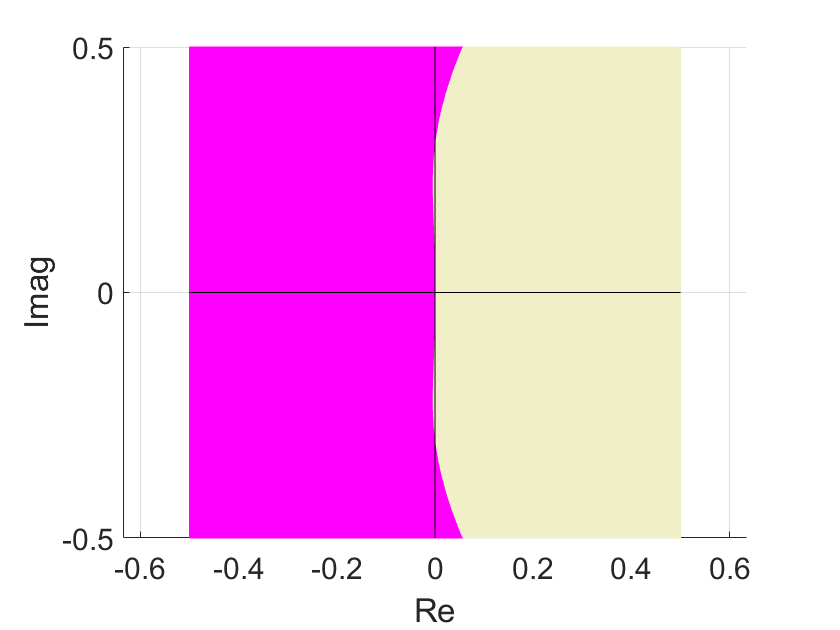}}\\
 \caption{ The pink parts show the region of  absolute stability of the general third order BDF scheme with Taylor expansion at $n+\beta, \, \beta=1,3,5$.
 }\label{fig:stability3}
 \end{figure}

\begin{figure}[!htbp]
 \centering
  \subfigure[fourth order, $\beta=1$]{ \includegraphics[scale=.22]{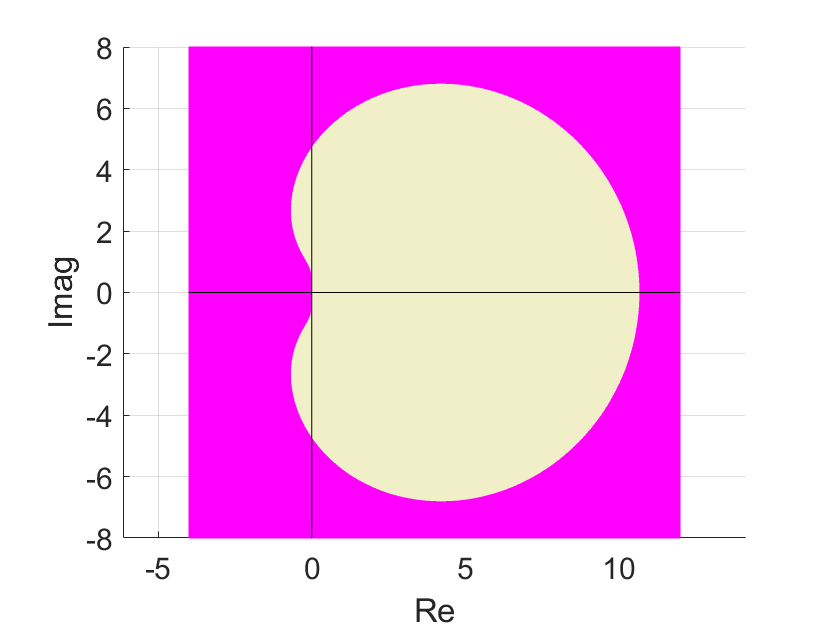}}
  \subfigure[fourth order, $\beta=3$]{ \includegraphics[scale=.22]{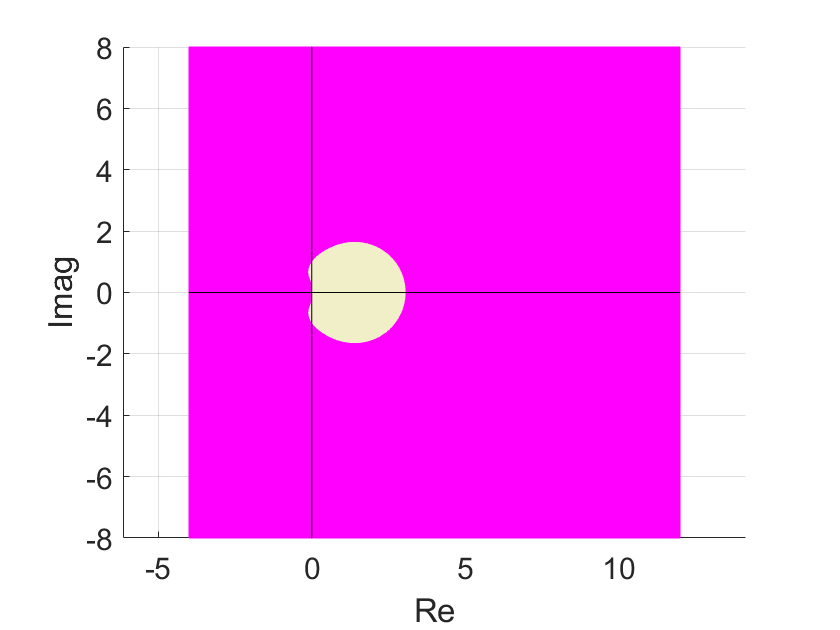}}
  \subfigure[fourth order, $\beta=5$]{ \includegraphics[scale=.22]{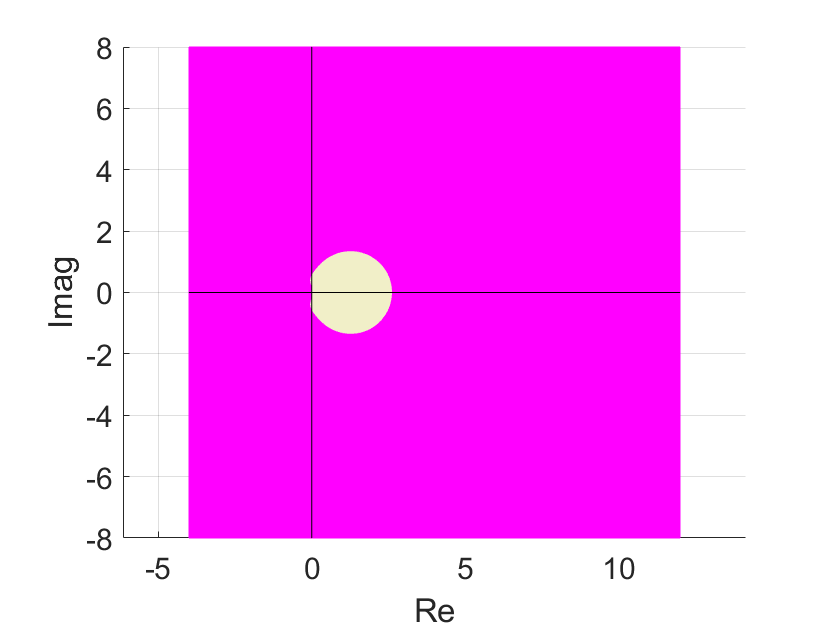}}\\
  \subfigure[$\beta=1$, zoom in around the origin]{ \includegraphics[scale=.22]{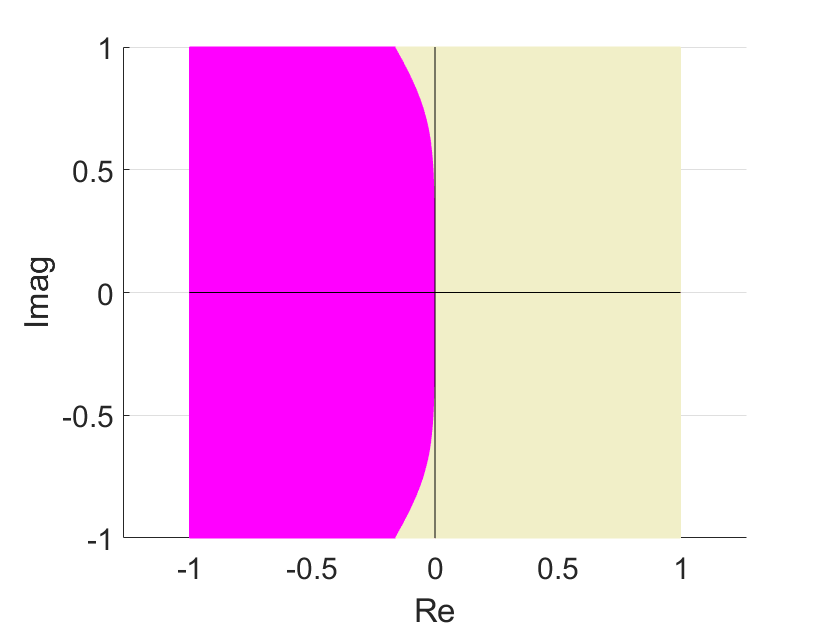}}
  \subfigure[$\beta=3$, zoom in around the origin]{ \includegraphics[scale=.22]{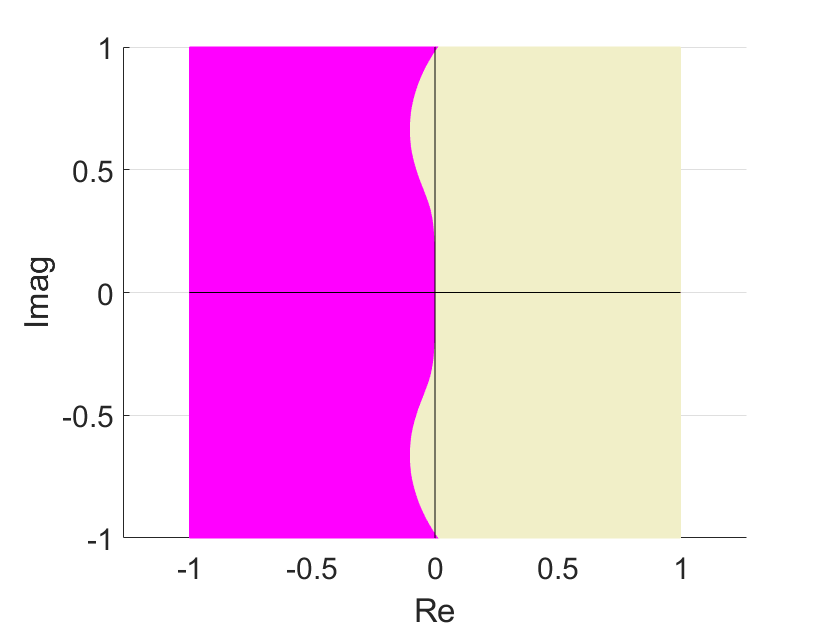}}
  \subfigure[$\beta=5$, zoom in around the origin]{ \includegraphics[scale=.22]{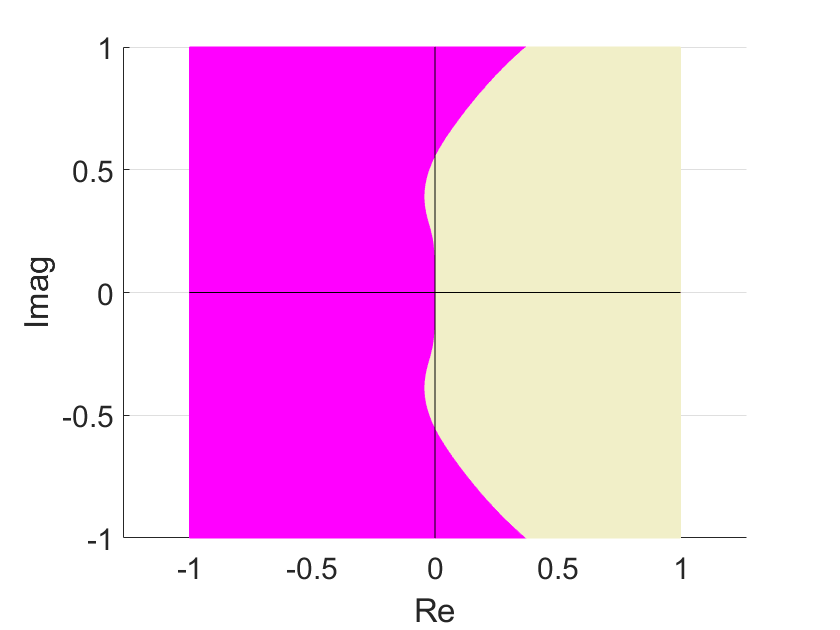}}\\
 \caption{ The pink parts show the region of  absolute stability of the general fourth-order BDF scheme with Taylor expansion at $n+\beta, \, \beta=1,3,5$.
 }\label{fig:stability4}
 \end{figure}

In order to have a better sense on how the stability regions vary with different $\beta$ and $k$, we plot in Table \ref{comparison} a comparison of stability regions in the same scale. We observe that (i) the stability regions increase faster when $\beta$ is closer to 1; and (ii) {\color{black} the area of}  the stability region with $k=4$ and $\beta=3$ is already bigger than that of the classical second-order BDF. Hence,  we can expect that the general fourth-order scheme with $\beta=3$ allows  similar or larger time steps for nonlinear problems than  the classical second-order IMEX, avoiding the usual scenario that smaller time step has to be used when increasing the accuracy order.

\begin{table}[h]
\centering
\begin{tabular}{c|ccc}
& $\beta=1$ &$\beta=3$ & $\beta=5$\\
\hline
second order& \begin{minipage}[b]{0.25\columnwidth}
		\centering
		\raisebox{-.5\height}{\includegraphics[width=0.95\linewidth]{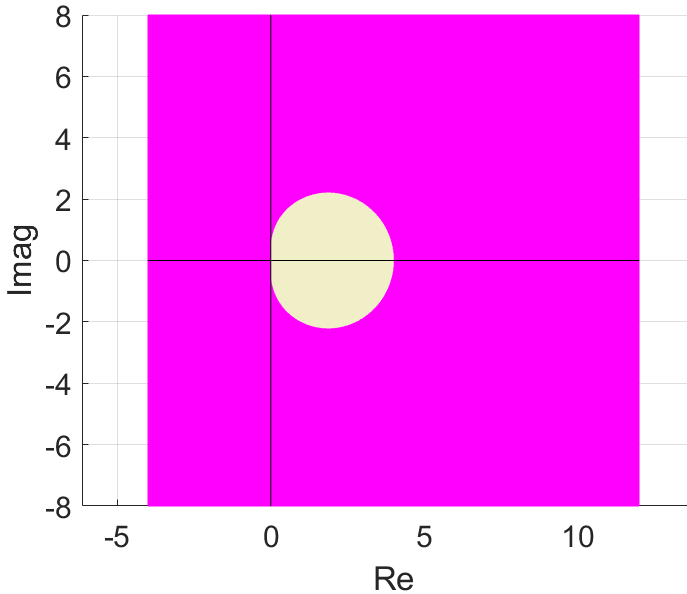}}
	\end{minipage}
&\begin{minipage}[b]{0.25\columnwidth}
		\centering
		\raisebox{-.5\height}{\includegraphics[width=0.95\linewidth]{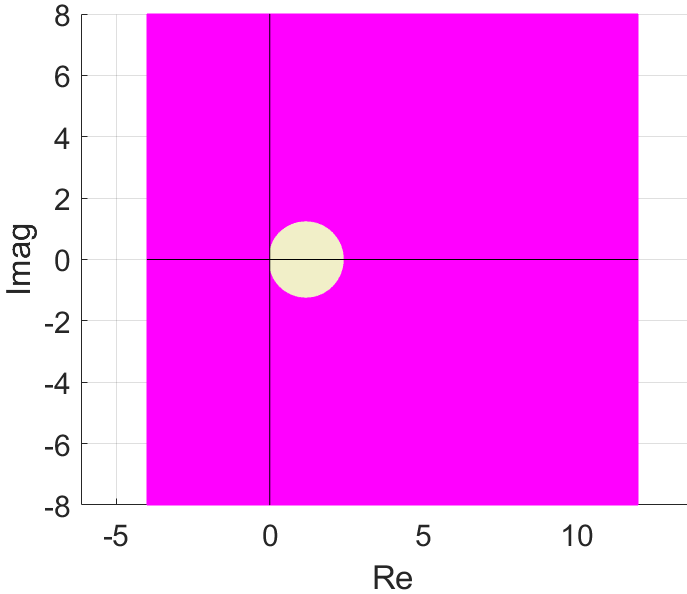}}
	\end{minipage}
&\begin{minipage}[b]{0.25\columnwidth}
		\centering
		\raisebox{-.5\height}{\includegraphics[width=0.95\linewidth]{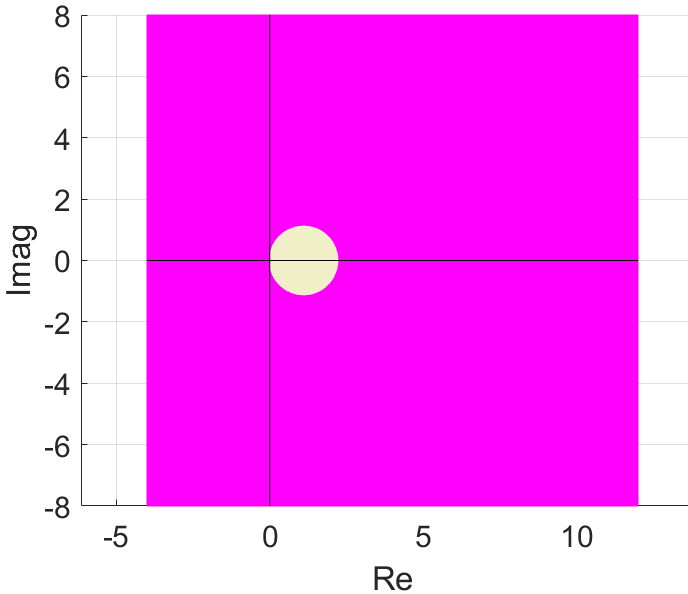}}
	\end{minipage}\\
\hline
third order & \begin{minipage}[b]{0.25\columnwidth}
		\centering
		\raisebox{-.5\height}{\includegraphics[width=0.95\linewidth]{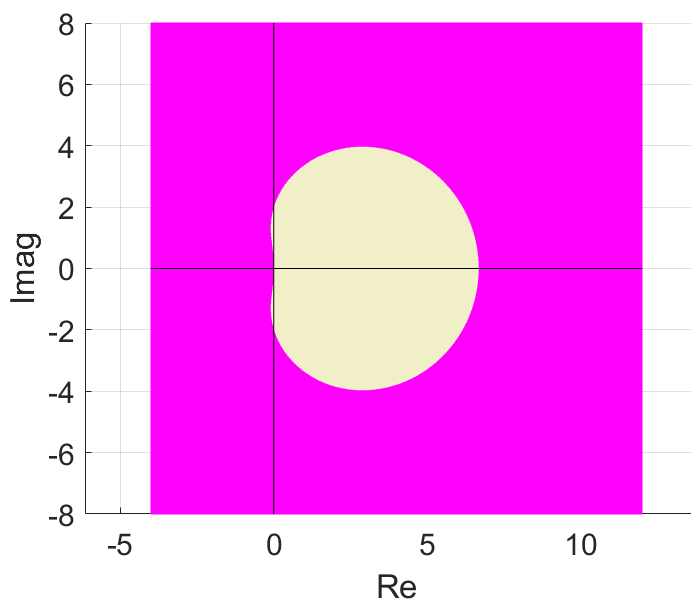}}
	\end{minipage}
&\begin{minipage}[b]{0.25\columnwidth}
		\centering
		\raisebox{-.5\height}{\includegraphics[width=0.95\linewidth]{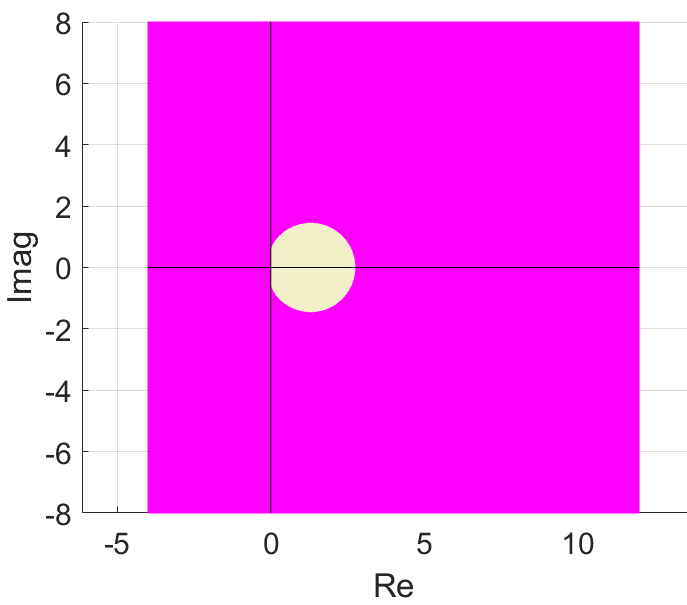}}
	\end{minipage}
&\begin{minipage}[b]{0.25\columnwidth}
		\centering
		\raisebox{-.5\height}{\includegraphics[width=0.95\linewidth]{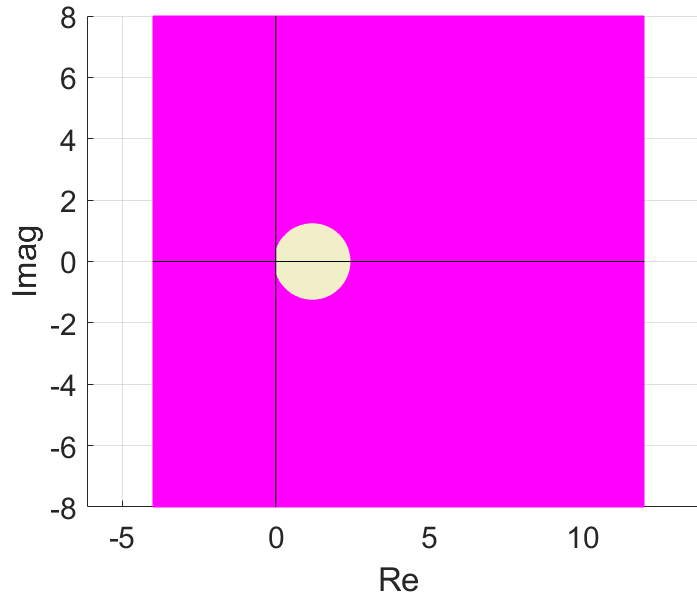}}
	\end{minipage}\\
\hline
fourth order & \begin{minipage}[b]{0.25\columnwidth}
		\centering
		\raisebox{-.5\height}{\includegraphics[width=0.95\linewidth]{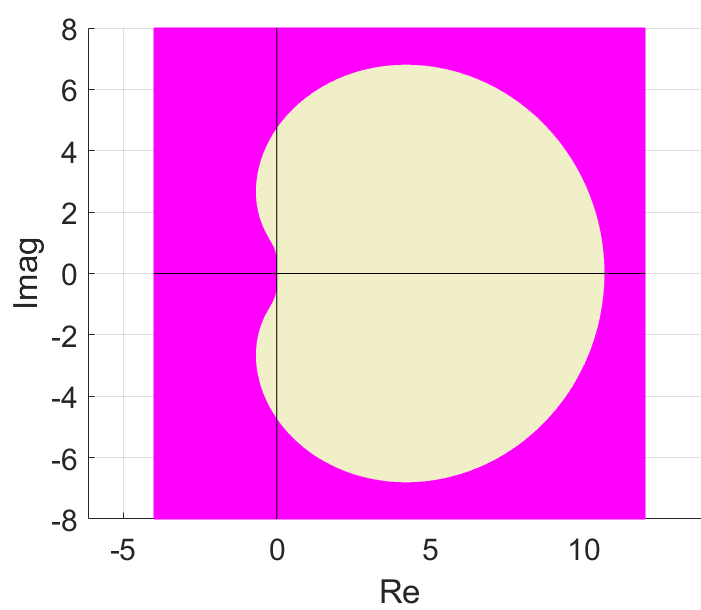}}
	\end{minipage}
&\begin{minipage}[b]{0.25\columnwidth}
		\centering
		\raisebox{-.5\height}{\includegraphics[width=0.95\linewidth]{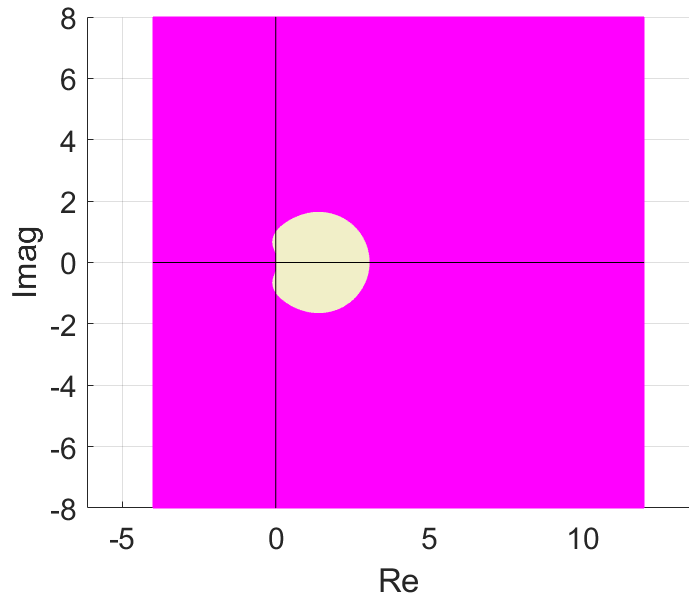}}
	\end{minipage}
&\begin{minipage}[b]{0.25\columnwidth}
		\centering
		\raisebox{-.5\height}{\includegraphics[width=0.95\linewidth]{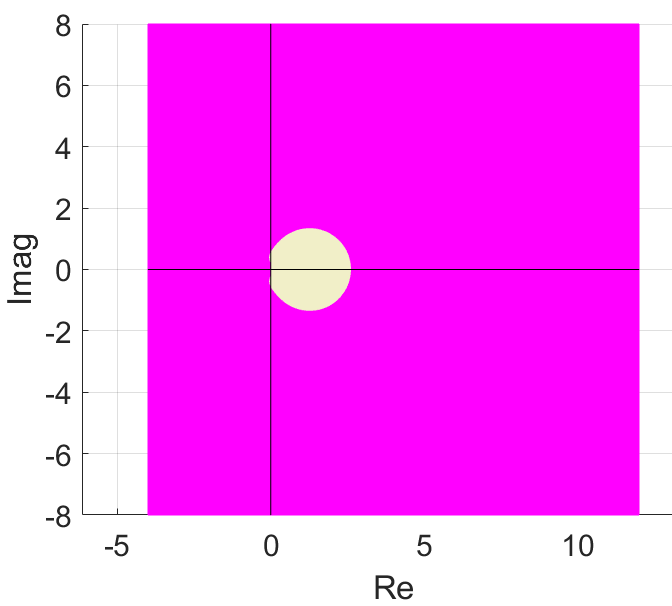}}
	\end{minipage}\\
\hline
\end{tabular}
\caption{Comparison of stability regions for different $k$ and $\beta$ on the same scale. }\label{comparison}
\end{table}

\section{Multipliers for the new BDF and IMEX  schemes}\label{multiplier}
In order to conduct the stability and error analysis for the BDF and IMEX  schemes by using energy techniques, a key step is  to find a suitable multiplier.  A key result which allows one to prove energy stability of  the classical BDF schemes of order up to five  is established in \cite{nevanlinna1981} where the existence of such multiplier is shown, see \cite{BDF6} for extension of this result to six-order BDF.  In this section, we identify  an explicit multiplier, and show that it is  suitable for the  new  BDF and IMEX   schemes of second to fourth order.
\subsection{Notations and a key lemma}

To simplify the presentations, we introduce the following notations:
\begin{equation}\label{IMEX234coeff}
\quad A_k^{\beta}(\phi^{i})=\sum_{q=0}^{k}a_{k,q}(\beta)\phi^{i-k+q},\,\quad B_k^{\beta}(\phi^{i})=\sum_{q=0}^{k-1}b_{k,q}(\beta)\phi^{i-k+1+q},\,\quad C_k^{\beta}(\phi^{i})=\sum_{q=0}^{k-1}c_{k,q}(\beta) \phi^{i-k+1+q}.
\end{equation}
with $a_{k,q}, b_{k,q}, c_{k,q}$ defined in \eqref{IMEX2coeff}, \eqref{IMEX3coeff} and \eqref{IMEX4coeff}.
 We also consider the characteristic polynomials of the new  BDF and IMEX  schemes  \eqref{genBDF} and \eqref{genIMEX}:
\begin{subequations}\label{poly}
	\begin{align}
		& \tilde{A}_k^{\beta}(\zeta)=\sum_{q=0}^{k}a_{k,q}(\beta)\zeta^q,\quad k=2,3,4;\\
		& \tilde{C}_k^{\beta}(\zeta)=\sum_{q=0}^{k-1}c_{k,q}(\beta)\zeta^q,\quad k=2,3,4.
	\end{align}
\end{subequations}

We first recall the following result from Dahlquist's G-stability theory \cite{dahlquist1978g} which plays a key role in establishing energy stability of multistep methods.
\begin{lemma}\label{Gstability}
	Let $\alpha(\zeta)=\alpha_k\zeta^k+...+\alpha_0$ and $\mu(\zeta)=\mu_k\zeta^k+...+\mu_0$ be polynomials of degree at most $k$ (and at least one of them of degree $k$) that have no common divisors. Let $(\cdot, \cdot)$ be an inner product with associated norm $|\cdot|$. If
	\begin{equation}\label{alphamu}
		\text{Re}\frac{\alpha(\zeta)}{\mu(\zeta)}>0 \quad \text{for} \, |\zeta|>1,
	\end{equation}
	then there exists a symmetric positive definite matrix $G=(g_{ij}) \in \mathbb{R}^{k\times k}$ and real $\delta_0,...,\delta_k$ such that for $\upsilon^0,...,\upsilon^k$ in the inner product space,
	\begin{equation}
		\big(\sum_{i=0}^k \alpha_i \upsilon^i, \sum_{j=0}^k \mu_j \upsilon^j \big)=\sum_{i,j=1}^{k}g_{ij}(\upsilon^i,\upsilon^j)-\sum_{i,j=1}^{k}g_{ij}(\upsilon^{i-1},\upsilon^{j-1})+\big |\sum_{i=0}^k\delta_i \upsilon^i \big|^2.
	\end{equation}
	
\end{lemma}
It is clear from the above Lemma that the key for establishing the energy stability of \eqref{genIMEX}
is to find a suitable multiplier {\color{black} $\mu(\zeta)=\mu_k\zeta^k+...+\mu_0$} such that \eqref{alphamu} is satisfied with $\alpha(\zeta)= \tilde{A}_k^{\beta}(\zeta)$.
To this end, we first  split $B_k^{\beta}(\phi^{n+1})$ into two parts:
\begin{equation}\label{splitB}
B_k^{\beta}(\phi^{n+1})=\eta_k(\beta)C_k^{\beta}(\phi^{n+1})+D_k^{\beta}(\phi^{n+1}),\quad k=2,3,4,
\end{equation}
with
\begin{equation}\label{eta}
\eta_2(\beta)=\frac{\beta-1}{\beta},\quad \eta_3(\beta)=\frac{\beta-1}{\beta+1},\quad \eta_4(\beta)=\frac{\beta-1}{\beta+3},\quad \beta \ge 1,
\end{equation}
and $D_k^{\beta}$ can be written as
\begin{equation}
	D^{\beta}_k(\phi^{n+1})=\sum_{q=0}^{k-1}d_{k,q}(\beta)\phi^{n+2-k+q},\quad k=2,3,4,
\end{equation}
with
\begin{subequations}
	\begin{align}
		& d_{2,1}(\beta)=\frac{1}{\beta},\,d_{2,0}(\beta)=0,\\
		& d_{3,2}(\beta)=1,\,d_{3,1}(\beta)=\frac{1-\beta}{1+\beta},\,d_{3,0}(\beta)=0,\\
		& d_{4,3}(\beta)=\frac{\beta^2}{6}+\frac{\beta}{2}+\frac{1}{3},\,d_{4,2}(\beta)=-(\frac{\beta^2}{2}+\frac{\beta}{2}-1),\,d_{4,1}(\beta)=\frac{\beta(\beta-1)}{2},\,d_{4,0}(\beta)=-\frac{\beta(\beta^2-1)}{6(\beta+3)}.
	\end{align}
\end{subequations}
We also define
\begin{align}
		 \tilde{D}_k^{\beta}(\zeta)=\sum_{q=0}^{k-1}d_{k,q}(\beta)\zeta^q,\quad k=2,3,4.
\end{align}
\begin{rem} The choices of $\eta_i(\beta)$ are not unique. We choose $\eta_2(\beta),\,\eta_3(\beta)$ defined in \eqref{eta} to make $D_2^{\beta},\,D_3^{\beta}$ as simple as possible and the choice of $\eta_4(\beta)$  defined in \eqref{eta} allows us to prove \eqref{IMEX4D} in the next subsection.
\end{rem}


\subsection{A uniform multiplier}

Note that in \cite{nevanlinna1981}, it was shown that there exists a multiplier in the form of $\phi^{n+1}-\tilde{\eta}_k \phi^n$ with $\tilde{\eta}_k \ge 0$ for the usual BDF schemes of order 2 to 5. Surprisingly, we can find a uniform multiplier for the new BDF and IMEX  schemes of order 2 to 4. More precisely, we have the following results.

\begin{theorem}\label{Thm1}
Given  $\beta \ge 1$, then
\begin{equation}
\rm{gcd}\big(\tilde{A}^{\beta}_k(\zeta),\zeta\tilde{C}^{\beta}_k(\zeta)\big)=\rm{gcd}\big(\tilde{D}^{\beta}_k(\zeta),\tilde{C}^{\beta}_k(\zeta)\big)=1, \, k=2, 3, 4,
\end{equation}
 i.e. they have no common divisor, and
\begin{equation}\label{IMEXA}
\text{Re}\frac{\tilde{A}^{\beta}_k(\zeta)}{\zeta\tilde{C}^{\beta}_k(\zeta)}>0, \quad \text{for}\, |\zeta|>1,\,  k=2,3,4.
\end{equation}
Moreover, we also have
\begin{equation}\label{IMEX23D}
\text{Re}\frac{\tilde{D}_k^{\beta}(\zeta)}{\tilde{C}_k^{\beta}(\zeta)}>0, \quad \text{for}\, |\zeta|>1,\, k=2,3;
\end{equation}
and finally if $\beta \ge 2$, then we also have
\begin{equation}\label{IMEX4D}
\text{Re}\frac{\tilde{D}^{\beta}_4(\zeta)}{\tilde{C}_4^{\beta}(\zeta)}>0, \quad \text{for}\, |\zeta|>1.
\end{equation}
\end{theorem}
\begin{proof}
The proof follows the basic process in \cite{BDF6}. We will  provide the proof for the case $k=4$ in detail as it includes some technical estimations and then we will point out the key steps for the cases $k=2,3$, which are easier to handle. To simplify the notation, we often omit the dependence on $\beta$ for the coefficients $a_{k,q}(\beta),c_{k,q}(\beta),d_{k,q}(\beta)$, i.e., we only write them as $a_{k,q}, c_{k,q}, d_{k,q}$.

\textbf{Case I: $k=4$.} Firstly, we show $\rm{gcd}\big(\tilde{A}_4^{\beta}(\zeta),\zeta\tilde{C}_4^{\beta}(\zeta)\big)=1$  by using the \textbf{Sylvester Resultant}  \cite{akritas1993sylvester} as follows. The Sylvester matrix  \cite{akritas1993sylvester} of $\tilde{A}_4^{\beta}(\zeta)$ and $\tilde{C}_4^{\beta}(\zeta)$ is
\begin{equation}
Sly(\tilde{A}_4^{\beta},\tilde{C}_4^{\beta})={\left(
\begin{array}{ccccccc}
a_{4,4}&a_{4,3}&a_{4,2}&a_{4,1}&a_{4,0}& 0 &0\\
0   &a_{4,4}&a_{4,3}&a_{4,2}&a_{4,1}&a_{4,0}&0\\
0&0&a_{4,4}&a_{4,3}&a_{4,2}&a_{4,1}&a_{4,0}\\
c_{4,3}&c_{4,2}&c_{4,1}&c_{4,0}&0&0&0\\
0&c_{4,3}&c_{4,2}&c_{4,1}&c_{4,0}&0&0\\
0&0&c_{4,3}&c_{4,2}&c_{4,1}&c_{4,0}&0\\
0&0&0&c_{4,3}&c_{4,2}&c_{4,1}&c_{4,0}
\end{array}
\right)}.
\end{equation}
It is easy to verify that  its determinant is
\begin{equation}
\text{det} Sly(\tilde{A}^{\beta}_4,\tilde{C}^{\beta}_4)=- \frac{1}{5184}\big(18\beta^6 + 144\beta^5 + 426\beta^4 + 566\beta^3 + 321\beta^2 + 55\beta +3\big) \neq 0,\;\text{ for } \beta\ge 1,
\end{equation}
which implies that $\rm{gcd}\big(\tilde{A}_4^{\beta}(\zeta), \tilde{C}_4^{\beta}(\zeta)\big)=1$. Combined with $\tilde{A}_4^{\beta}(0)=a_{4,0} \neq 0$,  it also implies that $\tilde{A}_{4}^{\beta}(\zeta)$ and $\zeta \tilde{C}_4^{\beta}(\zeta)$ have no common divisor.

 Next, we show $\frac{\tilde{A}_4^{\beta}(\zeta)}{\zeta \tilde{C}_4^{\beta}(\zeta)}$ is holomorphic outside the unit disk in the complex plane. To this end, it  suffices to show that  all three zeros of $\tilde{C}_4^{\beta}(\zeta)$ are inside the unit disk. Note that
\begin{equation}
\frac{d \tilde{C}_4^{\beta}}{d x}(x)=3c_{4,3}x^2+2c_{4,2}x+c_{4,1},\,
\end{equation}
with
\begin{equation}
 c_{4,3}=\frac{\beta^3+6\beta^2+11\beta+6}{6}>0,\, \Delta_4:=4c_{4,2}^2-12c_{4,3}c_{4,1}=-\beta(\beta + 2)(\beta + 3)^2<0,
\end{equation}
which means $\tilde{C}_4(x)$ is monotonically increasing in the real axis. Note also that
\begin{equation}\label{eq1}
\tilde{C}_4^{\beta}(0)=c_{4,0}=-\frac{\beta^3+3\beta^2+2\beta}{6}<0,\quad \tilde{C}^{\beta}_4(1)=c_{4,3}+c_{4,2}+c_{4,1}+c_{4,0}=1.
\end{equation}
Therefore, $\tilde{C}_4^{\beta}(\zeta)=0$ has exactly one real root, denoted as $x_1$, and two complex roots, denoted as $z_2,\, z_3=\bar{z}_2$, in the complex plane.  Next, we denote
\begin{equation}
x_0:=\frac{-c_{4,0}}{c_{4,3}-1}=\frac{\beta^2+3\beta+2}{\beta^2+6\beta+11}.
\end{equation}
Then we can find with $\beta \ge 1$,
\begin{equation}\label{eq2}
\tilde{C}_4^{\beta}(x_0)=-\frac{2\beta^6+27\beta^5+141\beta^4+351\beta^3+405\beta^2+162\beta-8}{(\beta^2+6\beta+11)^3}<0.
\end{equation}
Combining \eqref{eq1} and \eqref{eq2}, we have $x_0<x_1<1$. On the other hand, by Vieta's formulae, we have
\begin{equation}
x_1z_2z_3=x_1|z_2|^2=-\frac{c_{4,0}}{c_{4,3}},\quad \text{then} \,\, |z_2|^2=\frac{1}{x_1} \frac{-c_{4,0}}{c_{4,3}}<\frac{1}{x_0} \frac{-c_{4,0}}{c_{4,3}}=\frac{c_{4,3}-1}{c_{4,3}}<1.
\end{equation}
As a result, we have $|x_1|, |z_2|, |z_3| <1$ and hence $\frac{\tilde{A}_4^{\beta}(\zeta)}{\zeta \tilde{C}_4^{\beta}(\zeta)}$ and $\frac{\tilde{D}_4^{\beta}(\zeta)}{\tilde{C}_4^{\beta}(\zeta)}$ are holomorphic outside the unit disk.

On the other hand, we have
\begin{equation}
\lim_{|\zeta| \rightarrow \infty}\frac{\tilde{A}_4^{\beta}(\zeta)}{\zeta \tilde{C}_4^{\beta}(\zeta)}=\frac{a_{4,4}}{c_{4,3}}=\frac{2\beta^3+9\beta^2+11\beta+3}{2(\beta^3+6\beta^2+11\beta+6)}>0.
\end{equation}
Therefore, it follows from the maximum principle for harmonic functions, $\text{Re}\frac{\tilde{A}_4^{\beta}(\zeta)}{\zeta\tilde{C}_4^{\beta}(\zeta)}>0,\,\forall |\zeta|>1 $ is equivalent to
\begin{equation}
\text{Re} \frac{\tilde{A}_4^{\beta}(\zeta)}{\zeta \tilde{C}_4^{\beta}(\zeta)}\ge 0, \quad \forall \zeta \in \mathbb{S}^1,
\end{equation}
with $\mathbb{S}^1$ being the unit circle in the complex plane, and  which is equivalent to
\begin{equation}\label{A11}
Re[\tilde{A}_4^{\beta}(e^{i\theta})e^{-i\theta}\tilde{C}_4^{\beta}(e^{-i\theta})] \ge 0, \quad \theta \in [0, 2\pi).
\end{equation}
Letting $y:=\cos(\theta)$ and using the trigonometric identities
\begin{equation}
\cos(2\theta)=2y^2-1,\, \cos(3\theta)=4y^3-3y,\, \sin(2\theta)=2y\sin(\theta),\,\sin(3\theta)=(4y^2-1)\sin(\theta),
\end{equation}
 we find
\begin{equation}\label{eq3}
\begin{split}
\tilde{C}_4^{\beta}(e^{-i\theta})&=c_{4,3}\cos(3\theta)+c_{4,2}\cos(2\theta)+c_{4,1}\cos(\theta)+c_{4,0}-i[c_{4,3}\sin(3\theta)+c_{4,2}\sin(2\theta)+c_{4,1}\sin(\theta)]\\
&=c_{4,3}(4y^3-3y)+c_{4,2}(2y^2-1)+c_{4,1}y+c_{4,0}-i[c_{4,3}(4y^2-1)+2c_{4,2}y+c_{4,1}]\sin(\theta),
\end{split}
\end{equation}
and
\begin{equation}\label{eq4}
\begin{split}
\tilde{A}_4^{\beta}(e^{i\theta})e^{-i\theta}&=a_{4,4}e^{3i\theta}+a_{4,3}e^{2i\theta}+a_{4,2}e^{i\theta}+a_{4,1}+a_{4,0}e^{-i\theta}\\
&=a_{4,4}\cos(3\theta)+a_{4,3}\cos(2\theta)+a_{4,2}\cos(\theta)+a_{4,1}+a_{4,0}\cos(\theta)\\
&+i[a_{4,4}\sin(3\theta)+a_{4,3}\sin(2\theta)+a_{4,2}\sin(\theta)-a_{4,0}\sin(\theta)]\\
&=a_{4,4}(4y^3-3y)+a_{4,3}(2y^2-1)+(a_{4,2}+a_{4,0})y+a_{4,1}\\
& +i[a_{4,4}(4y^2-1)+2a_{4,3}y+a_{4,2}-a_{4,0}]\sin(\theta).
\end{split}
\end{equation}
It follows from \eqref{eq3}, \eqref{eq4} and  $\tilde{A}_4^{\beta}(1)=0$, $\sin^2(\theta)=1-y^2$ that
\begin{equation}
Re[\tilde{A}_4^{\beta}(e^{i\theta})e^{-i\theta}\tilde{C}_4^{\beta}(e^{-i\theta})]=\frac{1}{9}(1-y)(\omega_3(\beta)y^3+\omega_2(\beta)y^2+\omega_1(\beta)y+\omega_0(\beta))=:\frac{1}{9}(1-y)f_4(y),
\end{equation}
with
\begin{equation}\label{eq5}
\begin{split}
& f_4(y)=\omega_3(\beta)y^3+\omega_2(\beta)y^2+\omega_1(\beta)y+\omega_0(\beta),\\
& \omega_0(\beta)= 2\beta^6+15\beta^5+39\beta^4+39\beta^3+10\beta^2+15,\\
& \omega_1(\beta)=-6\beta^6-45\beta^5-117\beta^4-116\beta^3-21\beta^2+17\beta+9,\\
& \omega_2(\beta)=6\beta^6+45\beta^5+117\beta^4+115\beta^3+12\beta^2-34\beta-12,\\
& \omega_3(\beta)=-2\beta^6-15\beta^5-39\beta^4-38\beta^3-\beta^2+17\beta+6.
\end{split}
\end{equation}
In the following, we omit the dependence on $\beta$ for $\omega_i,\, i=0,1,2,3$.

It is clear that  \eqref{A11} is equivalent to
\begin{equation}\label{eq6}
f_4(y) \ge 0,\quad \forall y \in[-1,1].
\end{equation}
With $\omega_i$ defined in \eqref{eq5} and $\beta \ge 1$, we have
\begin{equation}\label{eq7}
\begin{split}
f_4(1)&=\omega_0+\omega_1+\omega_2+\omega_3=18>0,\\
f_4(-1)&=\omega_0-\omega_1+\omega_2-\omega_3=16\beta^6 + 120\beta^5 + 312\beta^4 + 308\beta^3 + 44\beta^2 - 68\beta - 12>0,
\end{split}
\end{equation}
and
\begin{equation}
f'_4(y)=3\omega_3y^2+2\omega_2y+\omega_1.
\end{equation}
If $f_4'(y)$ does not have zero in $[-1,1]$, then \eqref{eq7} implies \eqref{eq6}. Otherwise, supposing there exists $-1 \le y_0 \le 1$ such that $f_4'(y_0)=0$, we only need to show $f_4(y_0) \ge 0$.  Indeed, with $f_4'(y_0)=0$,  we have
\begin{equation}
3f_4(y_0)=3f_4(y_0)-y_0f_4'(y_0)=\omega_2 y_0^2+2\omega_1 y_0+3\omega_0.
\end{equation}
Denote
\begin{equation}
g_4(y):=\omega_2 y^2+2\omega_1 y+3\omega_0;
\end{equation}
then with $\beta \ge 1$, we have
\begin{small}\begin{equation}
\begin{split}
g_4(1)& =\omega_2+2\omega_1+3\omega_0=51>0,\\
g_4(-1) &= \omega_2-2\omega_1+3\omega_0= 24\beta^6 +180\beta^5 +468\beta^4 + 464\beta^3+ 84\beta^2 - 68\beta + 15>0,\\
\Delta_g& =4\omega_1^2-12\omega_2\omega_0=- 1220\beta^6 - 9108\beta^5 - 23408\beta^4 - 22212\beta^3 - 1076\beta^2 + 7344\beta + 2484<0,
\end{split}
\end{equation}\end{small}which means $g_4(y)>0, \forall y \in[-1,1]$. In particular, we have $f_4(y_0)=\frac{1}{3}g_4(y_0)>0$ which implies \eqref{eq6}, which in turn
implies  \eqref{A11}. Therefore, we proved \eqref{IMEXA} with $k=4$.

Next, we prove \eqref{IMEX4D} with $\beta \ge 2$. The procedure is similar to the proof of  \eqref{IMEXA} above. First, the Sylvester matrix of  $\tilde{D}_4^{\beta}(\zeta)$ and $\tilde{C}_4^{\beta}(\zeta)$:
\begin{equation}
Sly(\tilde{D}_4^{\beta},\tilde{C}_4^{\beta})={\left(
\begin{array}{cccccc}
d_{4,3}&d_{4,2}&d_{4,1}&d_{4,0}& 0 &0\\
0   &d_{4,3}&d_{4,2}&d_{4,1}&d_{4,0}&0\\
0&0&d_{4,3}&d_{4,2}&d_{4,1}&d_{4,0}\\
c_{4,3}&c_{4,2}&c_{4,1}&c_{4,0}&0&0\\
0&c_{4,3}&c_{4,2}&c_{4,1}&c_{4,0}&0\\
0&0&c_{4,3}&c_{4,2}&c_{4,1}&c_{4,0}
\end{array}
\right)},
\end{equation}
and its determinant is
\begin{equation}\label{slyD4}
\text{det}\, Sly(\tilde{D}_4^{\beta},\tilde{C}_4^{\beta})=-\frac{\beta^2(\beta^2 + 3\beta + 2)^2}{36} <0
\end{equation}
which implies $\tilde{D}_4^{\beta}(\zeta)$ and $\tilde{C}_4^{\beta}(\zeta)$ have no common divisor. Since we have shown in the above that $\frac{\tilde{D}_4^{\beta}(\zeta)}{\tilde{C}_4^{\beta}(\zeta)}$ is holomorphic outside the unit disk, following the same process as above, we have that \eqref{IMEX4D} is equivalent to:
\begin{equation}\label{A22}
h_4(y)=\alpha_{3}(\beta)y^3+\alpha_{2}(\beta)y^2+\alpha_{1}(\beta)y+\alpha_{0}(\beta) \ge 0,\, \forall y \in [-1,1],
\end{equation}
with
\begin{small}
\begin{equation}
\begin{split}
&\alpha_{0}(\beta)=\frac{1}{9(\beta+3)}(2\beta^6+15\beta^5+35\beta^4+15\beta^3-37\beta^2-39\beta+9),\\
&\alpha_{1}(\beta)=- \frac{2\beta^5}{3} - 3\beta^4 - \frac{10\beta^3}{3} + \beta^2 + 2\beta + 1,\\
&\alpha_{2}(\beta)=\frac{1}{3}(\beta(2\beta^4 + 9\beta^3 + 12\beta^2+3\beta - 2)),\\
&\alpha_{3}(\beta)=-\frac{1}{9}(\beta(\beta + 1)^2(2\beta^2 +  5\beta +2)).
\end{split}
\end{equation}
\end{small}
In the following, we omit the dependence on $\beta$ for $\alpha_i, i=0,1,2,3$. Hence, we have
\begin{small}
\begin{equation}\label{eq8}
\begin{split}
&h_4(-1)=-\alpha_{3}+\alpha_{2}-\alpha_{1}+\alpha_{0}=\frac{2}{9(\beta+3)}(8\beta^6+60\beta^5+152\beta^4+132\beta^3-16\beta^2-57\beta-9)>0,\\
&h_4(1)=\alpha_{3}+\alpha_{2}+\alpha_{1}+\alpha_{0}=\frac{4}{\beta+3}>0,
\end{split}
\end{equation}
\end{small}
and
\begin{equation}\label{eq9}
h'_4(y)=3\alpha_3y^2+2\alpha_2y+\alpha_1.
\end{equation}
Similarly as before, if $h'_4(y)$ does not have zero in $[-1,1]$, then \eqref{eq8} implies \eqref{A22}. Suppose $-1 \le y_0 \le 1$ such that $h'_4(y_0)=0$, we only need to show $h_4(y_0) \ge 0$.

With $h_4'(y_0)=0$ and $\alpha_3 \ne 0$,  we have
\begin{equation}
\begin{split}
3h_4(y_0)=3h_4(y_0)-y_0h_4'(y_0)&=\alpha_2 y_0^2+2\alpha_1 y_0+3\alpha_0\\
&= \frac{\alpha_2}{3\alpha_3}h_4'(y_0)+(2\alpha_1-\frac{2\alpha_2^2}{3\alpha_3})y_0+3\alpha_0-\frac{\alpha_1\alpha_2}{3\alpha_3}\\
&= 0+(2\alpha_1-\frac{2\alpha_2^2}{3\alpha_3})y_0+3\alpha_0-\frac{\alpha_1\alpha_2}{3\alpha_3}.
\end{split}
\end{equation}
We define
\begin{equation}
p(y):=(2\alpha_1-\frac{2\alpha_2^2}{3\alpha_3})y+3\alpha_0-\frac{\alpha_1\alpha_2}{3\alpha_3}.
\end{equation}
Then $p(y^*)=0$ if we define $y^*$ as
\begin{equation}\label{eq10}
y^*:=\frac{\frac{\alpha_1\alpha_2}{3\alpha_3}-3\alpha_0}{2\alpha_1-\frac{2\alpha_2^2}{3\alpha_3}}=\frac{4\beta^4+30\beta^3+35\beta^2+3\beta}{4\beta^4+30\beta^3+71\beta^2+54\beta+9},
\end{equation}
and we also have
\begin{equation}
2\alpha_1-\frac{2\alpha_2^2}{3\alpha_3}=\frac{8\beta^2 + 28\beta + 6}{6\beta+3}>0.
\end{equation}
Therefore,  to prove \eqref{A22}, it  suffices to show $y_0\ge y^*$.
However, this is more complicated  as \eqref{eq10} implies that $y^*$ can be arbitrarily close to 1  by increasing $\beta$,  and meanwhile, there indeed exists $y^*<y_0<1$ such that $h_4'(y_0)=0$.

If follows from \eqref{eq9} that
\begin{equation}
y_0=\frac{-2\alpha_2 \pm \sqrt{\Delta_h}}{6\alpha_3},
\end{equation}
with
\begin{equation}
\Delta_h:=4\alpha_2^2-12\alpha_1\alpha_3=\frac{4\beta(\beta + 1)^2(4\beta^3 + 22\beta^2 + 31\beta +6)}{9}>0.
\end{equation}
We can estimate $\Delta_h$ as follows
\begin{equation}
\Delta_h<\Delta_h+\frac{4(\beta+ 1)^2(2\beta^3 + 5\beta^2 -6\beta)}{9}=\frac{4}{9}(\beta+1)^2(2\beta^2+6\beta)^2=:\Delta_h^*.
\end{equation}
To show $y_0 \ge y^*$, we only consider the smallest root of $h_4'(y)=0$. Since we have $\alpha_2>0$ and $\alpha_3<0$, the smallest root is
\begin{equation}\label{eq11}
y_0=\frac{-2\alpha_2 + \sqrt{\Delta_h}}{6\alpha_3}>\frac{-2\alpha_2 + \sqrt{\Delta_h^*}}{6\alpha_3}=\frac{2\beta^3 + 7\beta^2  + 3\beta - 8}{2\beta^3 + 7\beta^2 + 7\beta+2}.
\end{equation}
Finally, we can prove $y_0 \ge y^*$ as follows. It follows from \eqref{eq10} and \eqref{eq11} that
\begin{equation}
\begin{split}
y_0-y^*&>\frac{2\beta^3 + 7\beta^2  + 3\beta - 8}{2\beta^3 + 7\beta^2 + 7\beta+2}-\frac{4\beta^4+30\beta^3+35\beta^2+3\beta}{4\beta^4+30\beta^3+71\beta^2+54\beta+9}\\
&=\frac{ 56\beta^4 + 138\beta^3 - 95\beta^2 - 339\beta - 72}{8\beta^6 + 80\beta^5 + 300\beta^4 + 523\beta^3 + 430\beta^2 + 153\beta + 18},
\end{split}
\end{equation}
and given $\beta \ge 2$,
\begin{equation}\label{4Dp}
\begin{split}
56\beta^4 + 138\beta^3 - 95\beta^2 - 339\beta - 72 & \ge 56\times 2^3\beta+138\times 2\beta^2- 95\beta^2 - 339\beta - 72\\
&=109\beta+181\beta^2-72>0.
\end{split}
\end{equation}

Therefore, we have $y_0 \ge y^*$.  Hence \eqref{IMEX4D} is proved for $\beta \ge 2$.

For the case $k=2$ and $3$, we can prove \eqref{IMEXA} and \eqref{IMEX23D} by the same process as above, so we only point out some related facts below, which are sufficient to complete the proof.

\textbf{Case II: $k=2$.}
\begin{itemize}
\item \text{det}$\,Sly(\tilde{A}^{\beta}_2,\tilde{C}^{\beta}_2)=-\frac{1}{2} \neq 0$, \text{det}$\,Sly(\tilde{D}^{\beta}_2,\tilde{C}^{\beta}_2)=-1\neq 0,\,\tilde{A}^{\beta}_2(0) \neq 0$.
\item The only zero of $\tilde{C}_2^{\beta}(\zeta)$ is $\frac{\beta}{1+\beta}<1$, which means $\frac{\tilde{A}^{\beta}_2(\zeta)}{\zeta \tilde{C}^{\beta}_2(\zeta)}$ and $\frac{\tilde{D}^{\beta}_2(\zeta)}{\tilde{C}^{\beta}_2(\zeta)}$ are holomorphic outside the unit disk.
\item  For $k=2$, \eqref{IMEXA} is equivalent to
\begin{equation}
f_2(y)=(-2\beta^2-\beta+1)y+2\beta^2+\beta+1 \ge 0,\quad \forall y \in [-1,1],
\end{equation}
which is true since $f_2(y)$ is monotonically decreasing and $f_2(1)=2$.
\item For $k=2$, \eqref{IMEX23D} is equivalent to
\begin{equation}
h_3(y)=-y+1+\frac{1}{\beta} \ge 0, \quad \forall y \in [-1,1],
\end{equation}
which is obviously true.
\end{itemize}

\textbf{Case III: $k=3$.}
\begin{itemize}
\item \text{det}$\,Sly(\tilde{A}_3^{\beta},\tilde{C}_3^{\beta})=\frac{\beta^2}{8} + \frac{5\beta}{24} + \frac{1}{36}\neq 0$, \text{det}$\,Sly(\tilde{D}_3^{\beta},\tilde{C}_3^{\beta})=\frac{\beta(\beta+1)}{2}\neq 0,\,\tilde{A}^{\beta}_3(0) \neq 0,\, \forall \beta \ge 1.$
\item  $\tilde{C}_3^{\beta}(\zeta)$ has two complex zeros $z_1$ and $z_2$ such that $|z_1|^2=|z_2|^2=\frac{\beta}{\beta+2}<1$, which means $\frac{\tilde{A}_3^{\beta}(\zeta)}{\zeta \tilde{C}_3^{\beta}(\zeta)}$ and $\frac{\tilde{D}_3^{\beta}(\zeta)}{\tilde{C}_3^{\beta}(\zeta)}$ are holomorphic outside the unit disk.
\item  For $k=3$, \eqref{IMEXA} is equivalent to
\begin{equation}\label{IMEXA3}
f_3(y)=\sigma_2(\beta)y^2+\sigma_1(\beta)y+\sigma_0 \ge 0,\quad \forall y \in[-1,1],
\end{equation}
with
\begin{subequations}
\begin{align}
& \sigma_2(\beta)=3\beta^4 + 9\beta^3 + 5\beta^2 - 3\beta - 2,\\
& \sigma_1(\beta)=- 6\beta^4 - 18\beta^3 - 13\beta^2 + \beta + 4,\\
& \sigma_0(\beta)=3\beta^4 + 9\beta^3 + 8\beta^2 + 2\beta + 4.
\end{align}
\end{subequations}
\eqref{IMEXA3} is true since $\sigma_2(\beta)>0$ for  $\beta \ge 1$ and
\begin{subequations}
\begin{align}
& f_3(-1)= 12\beta^4 + 36\beta^3 + 26\beta^2 - 2\beta - 2>0,\\
& f_3(1)=6,\\
& \Delta_3:=\sigma_1^2-4\sigma_0\sigma_2=- 63\beta^4 - 186\beta^3 - 95\beta^2 + 72\beta + 48<0.
\end{align}
\end{subequations}
\item For $k=3$, \eqref{IMEX23D} is equivalent to
\begin{equation}\label{IMEXD3}
h_3(y)=\mu_2(\beta)y^2+\mu_1(\beta)y+\mu_0(\beta) \ge 0, \quad \forall y \in [-1,1],
\end{equation}
with
\begin{subequations}
\begin{align}
& \mu_2(\beta)=\beta(\beta+1),\\
& \mu_1(\beta)=- 2\beta^2 - 2\beta+ 1,\\
& \mu_0(\beta)=\frac{\beta^3 + 2\beta^2 + 1}{\beta + 1}.
\end{align}
\end{subequations}
\eqref{IMEXD3} is true since $\mu_2(\beta)>0$ for  $\beta \ge 1$ and
\begin{subequations}
\begin{align}
& h_3(-1)= \frac{2\beta(2\beta^2 + 4\beta + 1)}{\beta + 1}>0,\\
& h_3(1)=\frac{2}{\beta + 1}>0,\\
& \Delta_3^*:=\mu_1^2-4\mu_0\mu_2=1-8\beta <0.
\end{align}
\end{subequations}
\end{itemize}

The proof for all the cases is completed.
\end{proof}
\begin{rem}
The restriction $\beta\ge 2$  is  a sufficient condition for \eqref{IMEX4D}, which comes from  \eqref{4Dp}.  One can easily show  that \eqref{slyD4} and \eqref{4Dp} are true whenever $\beta>1.6$. On the other hand,  \eqref{A22} is not true when $\beta=1$ as $h_4(0.2)=-0.312<0$ with $h_4$ defined in \eqref{A22}. 
\end{rem}
\subsection{Explicit telescoping formulae for the second and third order schemes}
Note that Lemma \ref{Gstability} only provides the existence of a symmetric positive definite matrix $G$ without giving the exact value of $g_{ij}$. In the following, we  provide explicit formulae for $g_{ij}$ in the second and third order cases.

\begin{prop}
For the second-order version of \eqref{genIMEX}, we have
\begin{equation}
\big(D_2^{\beta}(\phi^{n+1}), C_2^{\beta}(\phi^{n+1}) \big)=\frac{1}{\beta}|\phi^{n+1}|^2+\frac{1}{2}|\phi^{n+1}|^2-\frac{1}{2}|\phi^{n}|^2+\frac{1}{2}|\phi^{n+1}-\phi^n|^2,
\end{equation}
and
\begin{equation}
\begin{split}
\big(A_2^{\beta}(\phi^{n+1}), C_2^{\beta}(\phi^{n+1}) \big)&=a_2|\phi^{n+1}|^2-a_2|\phi^{n}|^2+|b_2\phi^{n+1}+c_2\phi^n|^2-|b_2\phi^{n}+c_2\phi^{n-1}|^2\\
&+|d_2\phi^{n+1}+e_2\phi^n+f_2\phi^{n-1}|^2,
\end{split}
\end{equation}
where the coefficients are given by
\begin{equation*}
\begin{split}
&e_2=-\sqrt{2\beta(2\beta+1)},\,\Delta_2=2\beta(2\beta+1),\, c_2=f_2=\frac{-\sqrt{2}+\sqrt{\Delta_2}}{2},\\
&d_2=\sqrt{2}+f_2,\, E_2=-\beta(2\beta-1),\,b_2=\frac{E_2-2e_2f_2}{-2c},\,a_2=\frac{3\beta+1-2\sqrt{\beta(2\beta+1)}}{2(\beta+1)^2}.
\end{split}
\end{equation*}
Moreover, we have $a_2>0$ for all $\beta \ge 1$. 
\end{prop}

\begin{prop}

For the third-order version of \eqref{genIMEX}, we have
\begin{equation}\label{IMEX3a_hat}
\begin{split}
\big(D_3^{\beta}(\phi^{n+1}), C_3^{\beta}(\phi^{n+1}) \big)&=\hat{a}_3|\phi^{n+1}|^2-\hat{a}_3|\phi^{n}|^2+|\hat{b}_3\phi^{n+1}+\hat{c}_3\phi^n|^2-|\hat{b}_3\phi^{n}+\hat{c}_3\phi^{n-1}|^2\\
&+|\hat{d}_3\phi^{n+1}+\hat{e}_3\phi^n+\hat{f}_3\phi^{n-1}|^2,
\end{split}
\end{equation}
where the coefficients are given by
\begin{equation}\label{IMEX3_hat_coeff}
\begin{split}
&\hat{M}=\frac{2\beta^3 + 4\beta^2 + \beta + 1}{\beta + 1},\,\hat{N}=\frac{(2\beta^2 + 2\beta - 1)^2}{4},\, \hat{\Delta}_3=\hat{M}^2-4\hat{N}=\frac{4\beta(2\beta^2 + 4\beta + 1)}{(\beta + 1)^2}\\
&\hat{e}_3=-\sqrt{\frac{\hat{M}-\sqrt{\hat{\Delta}_3}}{2}},\,\hat{P}=\frac{\beta^3 + 2\beta^2 + 1}{\beta + 1}-\hat{e}_3^2,\,\hat{Q}=\frac{2\beta^3 + 4\beta^2 + \beta+1}{\beta + 1}-\hat{e}_3^2,\\
&\hat{f}_3=\frac{-\sqrt{\hat{P}}+\sqrt{\hat{Q}}}{2},\,\hat{c}_3=\hat{f}_3,\,\hat{d}_3=\sqrt{\hat{P}}+\hat{f}_3,\,\hat{b}_3=\frac{\beta(\beta-1)+4\hat{e}_3\hat{f}_3}{4\hat{c}_3},\,
\hat{a}_3=\frac{\beta^2}{2} + \frac{3\beta}{2} + 1-\hat{b}_3^2-\hat{d}_3^2;
\end{split}
\end{equation}
and
\begin{equation}\label{IMEX3a}
\begin{split}
&\big(A_3^{\beta}(\phi^{n+1}), C_3^{\beta}(\phi^{n+1}) \big)=a_3|\phi^{n+1}|^2-a_3|\phi^{n}|^2+|b_3\phi^{n+1}+c_3\phi^n|^2-|b_3\phi^{n}+c_3\phi^{n-1}|^2\\
&+|d_3\phi^{n+1}+e_3\phi^n+f_3\phi^{n-1}|^2
-|d_3\phi^{n}+e_3\phi^{n-1}+f_3\phi^{n-2}|^2+|g_3\phi^{n+1}+h_3\phi^{n}+i_3\phi^{n-1}+j_3\phi^{n-2}|^2,
\end{split}
\end{equation}
where the coefficients are given by
\begin{equation}\label{IMEX3_coeff}
\begin{split}
& M=2\beta^4 + 6\beta^3 + \frac{13\beta^2}{3} - \frac{\beta}{3} - \frac{1}{3},\,N=-(\frac{\beta^2}{2} - \frac{1}{6})(\frac{\beta^2}{2} + \frac{3\beta}{2} + 1),\,P=\frac{\sqrt{M}+1}{2},\\
&Q=-\frac{1}{2}(\beta(\frac{\beta^2}{2} - \frac{1}{6})(\beta + 1)),\,R=\beta^4 + \frac{7\beta^3}{2} + \frac{19\beta^2}{6} - \frac{\beta}{3} - 1,\,S=\frac{7\beta^4}{4} + \frac{25\beta^3}{4} + \frac{17\beta^2}{3} + \frac{\beta}{2} + \frac{1}{3},\\
& W=(\frac{\beta^2}{2} + \frac{3\beta}{2} + 1)(\frac{\beta^2}{2} + \beta + \frac{1}{3}),\,U=\frac{1}{2} - \frac{79\beta^2}{12} - \frac{21\beta^3}{4} - \frac{5\beta^4}{4} - \frac{23\beta}{12},\\
&f_3=\frac{\sqrt{P^2+2N}+P}{2},\, j_3=f_3,\,g_3=f_3-P,\,i_3=-\sqrt{M}-g_3,\,h_3=\sqrt{M}-f_3,\,e_3=\frac{2i_3j_3-Q}{2f_3},\\
&d_3=\frac{R-2g_3i_3}{2f_3},\,c_3=\sqrt{S-e_3^2-g_3^2-h_3^2},\,b_3=\frac{U-2d_3e_3-2g_3h_3}{2c_3},\,a_3=W-g_3^2-d_3^2-b_3^2.
\end{split}
\end{equation}
Moreover, it is numerically verified  that all variables appearing in \eqref{IMEX3_hat_coeff} and \eqref{IMEX3_coeff} are real and bounded, and   $\hat{a}_3,\, a_3>0$ for $1 \le \beta \le 100$ (cf.  Fig. \ref{fig:a3}).
\end{prop}

The proof of the above two propositions is based on the method of undetermined coefficients, more precisely,  we assume a desired form and use the method of undetermined coefficients to find the suitable coefficients. The detail of the proof is tedious but straightforward so we leave it to the interested readers.
\begin{figure}[!htbp]
 \centering
  \subfigure{ \includegraphics[scale=.3]{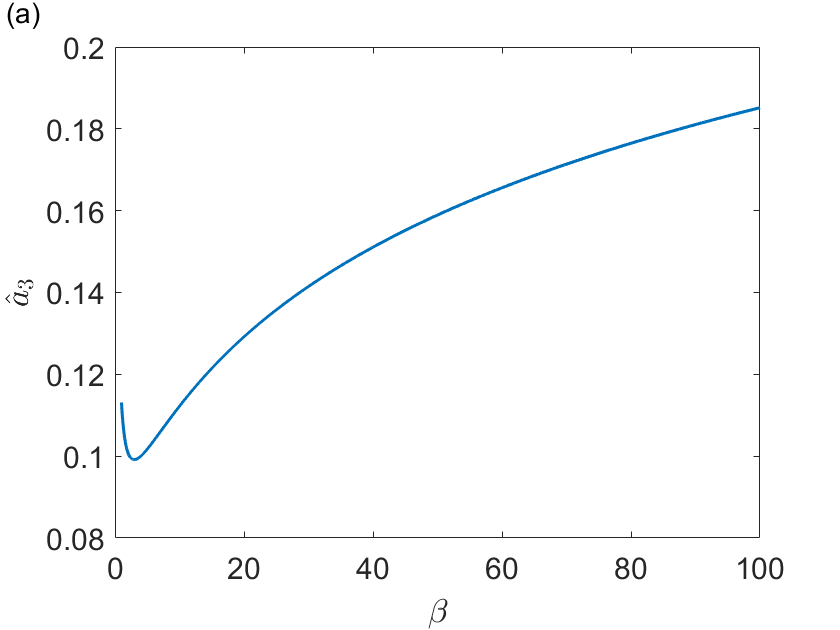}}
  \subfigure{ \includegraphics[scale=.3]{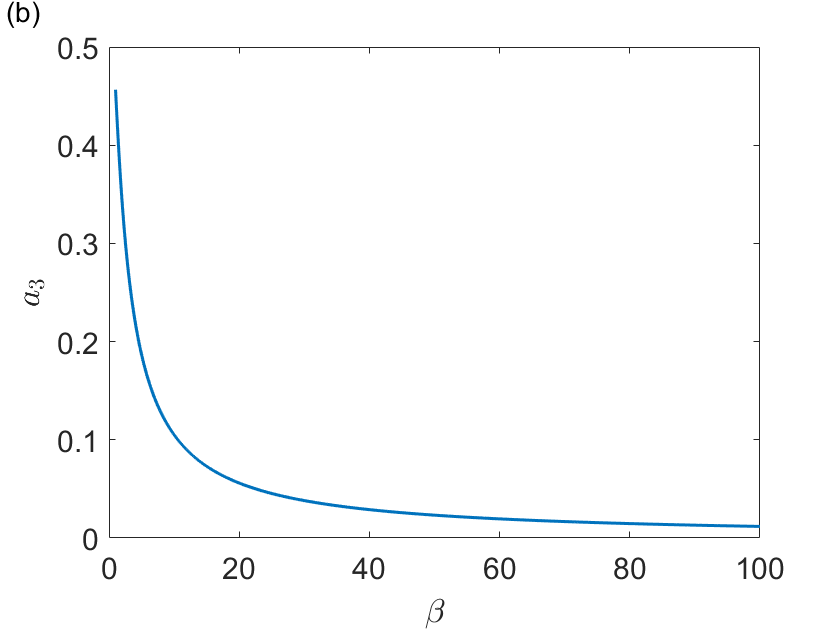}}
  \caption{ Values of $\hat{a}_3$ and $a_3$ with different $\beta$. }\label{fig:a3}
 \end{figure}


\section{Stability of \eqref{genBDF} for  linear parabolic type equations}
We consider in this section  the new BDF schemes for the linear case \eqref{genBDF}, which can be written as
\begin{equation}\label{linear_sche}
\frac{A_k^{\beta}(\phi^{n+1})}{\Delta t}+\mathcal{L}B_k^{\beta}(\phi^{n+1})=f^{n+\beta},\quad k=2,3,4,\end{equation}
and establish a stability result based on Theorem \ref{Thm1}.

\begin{theorem}\label{Thm_Stab1}
Assuming $\|f(t)\|_{\star}^2 \le C_f,\, \forall t\le T$, {\color{black}$\beta > 1$ for $k=2,3$, and  $\beta \ge 2$ for $k=4$, }then the   scheme \eqref{linear_sche}   is stable in the sense that
\begin{equation}\label{stab_linear0}
g_k|\phi^{n+1}|^2+\frac{1}{2}\Delta t \eta_k(\beta)\sum_{q=k}^{n+1}\|C_k^{\beta}(\phi^q)\|^2 \le C \sum_{q=0}^{k-1}\big(|\phi^q|^2+\Delta t \|\phi^q\|^2 \big)+\frac{TC_f}{2\eta_k(\beta)},\quad \forall k\le n+1\le \frac{T}{\Delta t},
\end{equation}
with $g_k$ a positive constant depending only on $k$, $C$ a constant independent of $\Delta t$ and $\eta_k(\beta)$ is defined in \eqref{eta}.
\end{theorem}
\begin{proof}
We denote $f^i=f(t^i),\,\forall i \le \frac{T}{\Delta t}$. Taking the inner product  of   \eqref{linear_sche} with $\Delta t C_k^{\beta}(\phi^{n+1})$ and splitting  $B_k^{\beta}(\phi^{n+1})$ as in \eqref{splitB}, we  obtain
\begin{small}
\begin{equation}\label{stab_linear1}
\big(A_k^{\beta}(\phi^{n+1}), C_k^{\beta}(\phi^{n+1})\big)+\Delta t \eta_k(\beta) \|C_k^{\beta}(\phi^{n+1})\|^2+\Delta t\big(\mathcal{L}D_k^{\beta}(\phi^{n+1}),C_k^{\beta}(\phi^{n+1})\big)=\Delta t\big(f^{n+\beta}, C_k^{\beta}(\phi^{n+1}) \big),
\end{equation}
\end{small}
where we used $\big(\mathcal{L}C_k^{\beta}(\phi^{n+1}),C_k^{\beta}(\phi^{n+1})\big)=\|C_k^{\beta}(\phi^{n+1})\|^2$. We estimate the terms in \eqref{stab_linear1} as follows.

It follows from \eqref{starnorm} and the assumption on $f$ that
\begin{equation}\label{stab_linear2}
\begin{split}
\big(f^{n+\beta}, C_k^{\beta}(\phi^{n+1}) \big) & \le \|f^{n+\beta}\|_{\star}\|C_k^{\beta}(\phi^{n+1})\|\\
& \le \frac{1}{2\eta_k(\beta)}\|f^{n+\beta}\|_{\star}^2+\frac{\eta_k(\beta)}{2}\|C_k^{\beta}(\phi^{n+1})\|^2\\
& \le \frac{C_f}{2\eta_k(\beta)}+\frac{\eta_k(\beta)}{2}\|C_k^{\beta}(\phi^{n+1})\|^2.\\
\end{split}
\end{equation}
Denote $\Phi^{n+1}_k:=(\phi^{n-k+1},...,\phi^{n+1})^{T}$. It follows from Lemma \ref{Gstability} and Theorem \ref{Thm1} that there exist symmetric positive definite matrices $G=(g_{ij}) \in \mathbb{R}^{k\times k} $ and $H=(h_{ij})\in \mathbb{R}^{(k-1)\times (k-1)}$ such that
\begin{equation}\label{Gdef}
\begin{split}
\big(A_k^{\beta}(\phi^{n+1}), C_k^{\beta}(\phi^{n+1})\big) &\ge \sum_{i,j=1}^k g_{ij}(\phi^{n+1+i-k},\phi^{n+1+j-k})-\sum_{i,j=1}^k g_{ij}(\phi^{n+i-k},\phi^{n+j-k})\\
&=: |\Phi^{n+1}_k|^2_G-|\Phi^{n}_k|^2_G,
\end{split}
\end{equation}
and
\begin{equation}\label{Hdef}
\begin{split}
\big(\mathcal{L}D_k^{\beta}(\phi^{n+1}), C_k^{\beta}(\phi^{n+1})\big) & \ge \sum_{i,j=1}^{k-1} h_{ij}(\mathcal{L}\phi^{n+2+i-k},\phi^{n+2+j-k})-\sum_{i,j=1}^{k-1} h_{ij}(\mathcal{L}\phi^{n+1+i-k},\phi^{n+1+j-k})\\
&=: \|\Phi^{n+1}_k\|^2_H-\|\Phi^{n}_k\|^2_H.
\end{split}
\end{equation}
Now, combining \eqref{stab_linear1}-\eqref{Hdef}, we obtain
\begin{equation}\label{linear_tosum}
|\Phi^{n+1}_k|^2_G-|\Phi^{n}_k|^2_G +\Delta t(\|\Phi^{n+1}_k\|^2_H-\|\Phi^{n}_k\|^2_H)+\frac{1}{2}\Delta t\eta_k(\beta)\|C_k^{\beta}(\phi^{n+1})\|^2
 \le \frac{C_f\Delta t }{2\eta_k(\beta)}.
\end{equation}
Summing up \eqref{linear_tosum} from $n=k-1$ to $n=m$, we obtain
\begin{equation}\label{linear_sum1}
 |\Phi^{m+1}_k|^2_G +\Delta t\|\Phi^{m+1}_k\|^2_H+\frac{1}{2}\Delta t\eta_k(\beta)\sum_{q=k-1}^{m}\|C_k^{\beta}(\phi^{q+1})\|^2
\le |\Phi^{k-1}_k|^2_G+\Delta t\|\Phi^{k-1}_k\|^2_H+\frac{T C_f}{2\eta_k(\beta)} .
\end{equation}
Let $g_k$ be the smallest eigenvalue of the matrix $G \in \mathbb{R}^{k,k}$, then we have
\begin{equation}
|\Phi^{m+1}_k|^2_G \ge g_k |\phi^{m+1}|^2,
\end{equation}
 and we can choose a constant $C$ large enough such that
\begin{subequations}\label{constant_C}
\begin{align}
& |\Phi^{k-1}_k|^2_G \le C\sum_{i=0}^{k-1}|\phi^i|^2,\\
& \Delta t\|\Phi^{k-1}_k\|^2_H \le C\Delta t \sum_{i=0}^{k-1}\|\phi^i\|^2.
\end{align}
\end{subequations}
Finally, combining \eqref{linear_sum1} and \eqref{constant_C} leads to
\begin{equation}\label{linear_sum2}
g_k|\phi^{m+1}|^2+\frac{1}{2}\Delta t\eta_k(\beta)\sum_{q=k-1}^{m}\|C_k^{\beta}(\phi^{q+1})\|^2 \le C\sum_{i=0}^{k-1}(|\phi^i|^2+\Delta t \|\phi^i\|^2)+\frac{TC_f}{2\eta_k(\beta)},
\end{equation}
which  implies \eqref{stab_linear0}.
\end{proof}
{\color{black}
\begin{rem}
Note that in order to  obtain \eqref{Hdef},   the linear operator $\mathcal{L}$ is required to be  self-adjoint while using the Nevanlinna-Odeh approach in \cite{nevanlinna1981} can also deal with $\mathcal{L}$ which is  not self-adjoint.
\end{rem}}
\section{Stability and error analysis of \eqref{genIMEX}  for  nonlinear parabolic type equations}
In this section, we use the stability result established in the last section to carry out a  stability and  error analysis of  \eqref{genIMEX}  for  nonlinear parabolic equations.

\subsection{Stability}
Under the local Lipschitz condition \eqref{LocalLp} on the nonlinear operator $\mathcal{G}$, we can derive a local stability result for \eqref{genIMEX} similarly as in the proof of the linear case (cf. Theorem \ref{Thm_Stab1}) if we further assume
\begin{equation}\label{localCK}
C_k^{\beta}(\phi^{n}) \in \mathcal{B}_{\phi(t^{n+\beta})},
\end{equation}
with $\beta > 1$ for $k=2,3$, and  $\beta \ge 2$ for $k=4$.
Note that formally  \eqref{localCK} must be true when $\Delta t$ small enough since $C_k^{\beta}(\phi^{n})$ is a $k$-th order approximation to $\phi(t^{n+\beta})$. We shall defer the rigorous  proof of  \eqref{localCK} to  subsection \ref{error}  by induction together with the error analysis.

\subsection{Truncation errors}
Using the notations introduced in previous sections, we define the truncation errors for $k=2,3,4$ as
\begin{subequations}\label{truncation0}
\begin{align}
& E_k^{n+1}:=\Delta t\phi_t(t^{n+\beta})-A_k^{\beta}(\phi(t^{n+1})),\\
& R_k^{n+1}:=\phi(t^{n+\beta})-B_k^{\beta}(\phi(t^{n+1})),\\
& P_k^{n}:=\phi(t^{n+\beta})-C_k^{\beta}(\phi(t^{n})).
\end{align}
\end{subequations}
It follows from \eqref{TaylorA}, \eqref{TaylorB} and \eqref{TaylorC} that
\begin{equation}
E_k^{n+1}=\mathcal{O}(\Delta t^{k+1}),\quad R_k^{n+1}=\mathcal{O}(\Delta t^{k}),\quad \quad P_k^{n}=\mathcal{O}(\Delta t^{k}).\\
\end{equation}
More precisely, one can verify
\begin{subequations}\label{truncation}
\begin{align}
& E_k^{n+1}=\frac{1}{k!}\sum_{q=0}^k a_{k,q}(\beta)\int_{t^{n+1+q-k}}^{t^{n+\beta}}(t^{n+1+q-k}-s)^k\phi^{(k+1)}(s)ds,\\
& R_k^{n+1}=\frac{1}{(k-1)!}\sum_{q=0}^{k-1} b_{k,q}(\beta)\int_{t^{n+2+q-k}}^{t^{n+\beta}}(t^{n+2+q-k}-s)^{k-1}\phi^{(k)}(s)ds,\\
& P_k^{n}=\frac{1}{(k-1)!}\sum_{q=0}^{k-1} c_{k,q}(\beta)\int_{t^{n+1+q-k}}^{t^{n+\beta}}(t^{n+1+q-k}-s)^{k-1}\phi^{(k)}(s)ds.
\end{align}
\end{subequations}
Therefore, under suitable regularity requirements, we have
\begin{equation}\label{truncationerror}
|E_k^{n+1}|^2 \le C (\Delta t)^{2k+2},\quad \|R_k^{n+1}\|^2 \le C (\Delta t)^{2k}, \quad \|P_k^{n}\|^2 \le C (\Delta t)^{2k},\quad \forall n+1 \le \frac{T}{\Delta t}.
\end{equation}
\subsection{Error estimate}\label{error}
We denote $e^m:=\phi^m-\phi(t^m)$, where $\phi(t^m)$ is the exact solution of \eqref{nonlinear} at time $t^m$, i.e.,
\begin{equation}\label{nonlinear_exact}
\phi_t(t^m)+\mathcal{L}\phi({t^m})+\mathcal{G}[\phi(t^m)]=f(t^m).
\end{equation}
We will  use the following discrete version of the Gronwall lemma \cite{quarteroni2008numerical}.
\begin{lemma}\label{Gron1}
Let $y^k,\,h^k,\,g^k,\,f^k$ be four nonnegative sequences satisfying
\begin{equation*}
y^n+\Delta t \sum_{k=0}^{n}h^k \le B+\Delta t \sum_{k=0}^{n}(g^ky^k+f^k)\; \text{ with }\; \Delta t \sum_{k=0}^{T/\Delta t} g^k \le M,\, \forall\, 0\le n \le T/\Delta t.
\end{equation*}
We assume $\Delta t\, g^k <1$ for all $k$, and let $\sigma=\max_{0\le k \le T/\Delta t}(1-\Delta t g^k)^{-1}$. Then
\begin{equation*}
y^n+\Delta t \sum_{k=1}^{n}h^k \le \exp(\sigma M)(B+\Delta t\sum_{k=0}^{n}f^k),\,\,\forall\, n\le T/\Delta t.
\end{equation*}
\end{lemma}
\begin{theorem}\label{Thmerror}
Assume  \eqref{LocalLp} and the solution of \eqref{nonlinear} is sufficiently smooth such that \eqref{truncationerror} is true, and the following stability condition
\begin{equation}\label{stabcond}
\eta_k(\beta)-\sqrt{\gamma} \ge \rho>0
\end{equation}
 is satisfied. Given $\phi^0=\phi(0) \in V$,  we assume  $\beta > 1$ for $k=2,3$, and  $\beta \ge 2$ for $k=4$, and
  that  $\phi^i, i=1,...,k-1,$ are computed with a proper initialization procedure such that
\begin{equation}\label{initial}
|\phi^i-\phi(t^i)|^2,\,\|\phi^i-\phi(t^i)\|^2 \le C (\Delta t)^{2k}, i=1,...,k-1, \,\text{and} \,\, C_k^{\beta}(\phi^{k-1})\in \mathcal{B}_{\phi(t^{k-1+\beta})};
\end{equation}
 then 
 for $\Delta t$ sufficiently small, we have
 \begin{equation}\label{localCK0}
 	C_k^{\beta}(\phi^{n+1}) \in \mathcal{B}_{\phi(t^{n+1+\beta})},\quad \forall n+1\le \frac{T}{\Delta t},
 \end{equation}
 and
\begin{equation}\label{error0}
g_k|e^{n+1}|^2+\frac{\rho}{2}\Delta t\sum_{q=k-1}^{n+1}\|C_k(e^{q})\|^2 \le C\exp \big((1-C\Delta t)^{-1}T \big) (\Delta t)^{2k},\quad \forall n+1\le \frac{T}{\Delta t},
\end{equation}
where $g_k$ is a positive constant depending only on $k$, $C$ is  a constant independent of $\Delta t$. 

\end{theorem}
\begin{proof}
We shall prove \eqref{localCK0} and \eqref{error0} by induction. Suppose we already have
\begin{equation}\label{localCK1}
C_k^{\beta}(\phi^{n}) \in \mathcal{B}_{\phi(t^{n+\beta})}, \quad \forall n\le m,
\end{equation}
and \eqref{error0} is satisfied with all $n\le m-1$, we need  to prove
\begin{equation}\label{localCK2}
C_k^{\beta}(\phi^{m+1}) \in \mathcal{B}_{\phi(t^{m+1+\beta})},
\end{equation}
and \eqref{error0}  is satisfied with all $n\le m$.

Subtracting \eqref{nonlinear_exact} with $m=n+\beta$ from \eqref{genIMEX} and multiplying by $\Delta t$, we obtain
\begin{equation}\label{error1}
A_k^{\beta}(e^{n+1})+\Delta t\mathcal{L}B_k^{\beta}(e^{n+1})=-\Delta t\big(\mathcal{G}[C_k^{\beta}(\phi^n)]-\mathcal{G}[\phi(t^{n+\beta})]\big)+E_k^{n+1} +\Delta t \mathcal{L}R_k^{n+1}
\end{equation}
where $E_k^{n+1},\,R_k^{n+1}$ are given in \eqref{truncation0}. We split $ \mathcal{G}[C_k^{\beta}(\phi^n)]-\mathcal{G}[\phi(t^{n+\beta})]$ as
\begin{equation}\label{P}
\begin{split}
\mathcal{G}[C_k^{\beta}(\phi^n)]-\mathcal{G}[\phi(t^{n+\beta})]&=\big(\mathcal{G}[C_k^{\beta}(\phi^n)]-\mathcal{G}[C_k^{\beta}(\phi(t^n))]\big)+\big(\mathcal{G}[C_k^{\beta}(\phi(t^n))]-\mathcal{G}[\phi(t^{n+\beta})]\big)\\
&=:T_1^n+T_2^n.
\end{split}
\end{equation}
Taking  the inner product of  \eqref{error1} with $C_k^{\beta}(e^{n+1})$, and splitting $B_k^{\beta}(e^{n+1})$ as in \eqref{splitB},  we obtain
\begin{equation}\label{error2}
\begin{split}
& \big(A_k^{\beta}(e^{n+1}), C_k^{\beta}(e^{n+1})\big)+\Delta t \eta_k(\beta) \|C_k^{\beta}(e^{n+1})\|^2+\Delta t\big(\mathcal{L}D_k^{\beta}(e^{n+1}),C_k^{\beta}(e^{n+1})\big)\\
& = -\Delta t\big(T_1^n, C_k^{\beta}(e^{n+1})\big)-\Delta t\big(T_2^n, C_k^{\beta}(e^{n+1})\big)+\big(E_k^{n+1}, C_k^{\beta}(e^{n+1}) \big)+\Delta t \big(\mathcal{L}R_k^{n+1}, C_k^{\beta}(e^{n+1}) \big).
\end{split}
\end{equation}
Next, we bound the right hand side of \eqref{error2} with the help of the  consistency estimate. First, it follows from \eqref{TaylorC} that with $\Delta t$ sufficiently small, we have
$C_k^{\beta}(\phi(t^n)) \in \mathcal{B}_{\phi(t^{n+\beta})}$, then for the terms with $T_1^n$ and $T_2^n$, it follows from \eqref{LocalLp} and \eqref{localCK1} that for any given $\varepsilon>0$,
\begin{equation}\label{T1}
\big |\big(T_1^n, C_k^{\beta}(e^{n+1})\big)\big| \le \|T_1^n\|_{\star}\|C_k^{\beta}(e^{n+1})\| \le \frac{\varepsilon}{2}(\gamma\|C_k^{\beta}(e^n)\|^2+\mu|C_k^{\beta}(e^n)|^2)+\frac{1}{2\varepsilon}\|C_k^{\beta}(e^{n+1})\|^2,
\end{equation}
 With $P_k^n$ defined in \eqref{truncation0}, we have
\begin{equation}\label{T2}
\begin{split}
\big|\big(T_2^n, C_k^{\beta}(e^{n+1})\big)\big| \le \|T_2^n\|_{\star}\|C_k^{\beta}(e^{n+1})\|  & \le  \frac{1}{\rho}(\gamma\|P_k^n\|^2+\mu|P_k^n|^2)+\frac{\rho}{4}\|C_k^{\beta}(e^{n+1})\|^2,\\
& \le C(\Delta t)^{2k}+\frac{\rho}{4}\|C_k^{\beta}(e^{n+1})\|^2.
\end{split}
\end{equation}
Similarly,
\begin{equation}\label{Ek}
\big(E_k^{n+1}, C_k^{\beta}(e^{n+1}) \big) \le \frac{1}{2\Delta t}|E_k^{n+1}|^2+\frac{\Delta t}{2}|C_k^{\beta}(e^{n+1})|^2 \le C (\Delta t)^{2k+1}+\frac{\Delta t}{2}|C_k^{\beta}(e^{n+1})|^2,
\end{equation}
and
\begin{equation}\label{Rk}
\big(\mathcal{L}R_k^{n+1}, C_k^{\beta}(e^{n+1}) \big) \le \frac{1}{\rho}\|R_k^{n+1}\|^2+\frac{\rho}{4}\| C_k^{\beta}(e^{n+1}) \|^2 \le C(\Delta t)^{2k}+\frac{\rho}{4}\|C_k^{\beta}(e^{n+1})\|^2.
\end{equation}
Now, under the stability condition \eqref{stabcond}, combining the assumption on the initial steps \eqref{initial} and estimations in  \eqref{T1}-\eqref{Rk}, taking $\varepsilon=\frac{1}{\sqrt{\gamma}}$ in \eqref{T1}, and following the same process as in the proof of Theorem \ref{Thm_Stab1} to handle the terms on the left hand side of \eqref{error2}, we can obtain the following from \eqref{error2}:
\begin{equation}\label{error3}
g_k|e^{n+1}|^2+\frac{\rho}{2}\Delta t\sum_{q=k-1}^{n+1}\|C_k^{\beta}(e^{q})\|^2 \le C \Delta t \sum_{q=0}^{n+1}|e^q|^2+C(\Delta t)^{2k},\quad \forall\, n\le m.
\end{equation}
Therefore, by applying the discrete Gronwall lemma \ref{Gron1} to \eqref{error3}, we can obtain
\begin{equation}\label{error4}
g_k|e^{m+1}|^2+\frac{\rho}{2}\Delta t\sum_{q=k-1}^{m+1}\|C_k^{\beta}(e^{q})\|^2 \le C\exp \big((1-C\Delta t)^{-1}T \big) (\Delta t)^{2k},\quad \forall  m+1\le \frac{T}{\Delta t},
\end{equation}
with $C$ a constant independent of $\Delta t$ which implies  \eqref{error0}. Finally, it follows from \eqref{error4} and \eqref{TaylorC} that
{\color{black}\begin{equation}
\begin{split}
\|C_k^{\beta}(\phi^{m+1})-\phi(t^{m+1+\beta})\|^2 &\le 2\|C_k^{\beta}(\phi^{m+1})-C_k^{\beta}(\phi(t^{m+1}))\|^2+2\|C_k^{\beta}(\phi(t^{m+1}))-\phi(t^{m+1+\beta})\|^2\\
&\le 2\|C_k^{\beta}(e^{m+1})\|^2+\mathcal{O}( \Delta t^{2k})\\
&\le \bar{C} \Delta t^{2k-1},
\end{split}
\end{equation}}with $\bar{C}$ a constant independent of $\Delta t$, which implies \eqref{localCK2} for $\Delta t$ sufficiently small.
Thus, the proof is complete with the induction.
\end{proof}
\begin{rem}
Note that $\eta_k(\beta)$ in \eqref{eta} monotonically increases as $\beta$ increases.
On the other hand, for many applications, given $\delta>0$, one can choose $\gamma<\delta$ with a suitable $\mu$ such that \eqref{LocalLp} is satisfied \cite{akrivis2015fully}. Hence, the stability condition \eqref{stabcond} can always be satisfied with these applications.
\end{rem}
\begin{rem}
{\color{black} The analysis in Theorem 2 and Theorem 3 can not be directly extended to the standard BDF methods (with $\beta=1$) since $\eta_k(1)=0$.}
\end{rem}
\subsection{Comparison to the classical BDF and IMEX schemes}
In this subsection, we compare the stability condition \eqref{stabcond}  to that of the classical BDF and IMEX methods (with Taylor expansion at time $t^{n+1}$) for which the stability condition \eqref{stabcond} does not apply. So we shall derive below a corresponding stability condition for the classical BDF and IMEX methods.  To simplify the presentation, we assume $\mu=0$ in \eqref{LocalLp} since the general case can be handled by applying the discrete Gronwall lemma as in Theorem \ref{Thmerror}.

The stability condition \eqref{stabcond} in Theorem \ref{Thmerror} is derived from
\begin{equation}\label{stabA2}
	\big(\mathcal{L}B_k^{\beta}(e^n), C_k^{\beta}(e^n) \big)=\big(\eta_k(\beta) C_k^{\beta}(e^n)+D_k^{\beta}(e^n), C_k^{\beta}(e^n)\big)=\eta_k(\beta)\|C_k^{\beta}(e^n)\|^2+\big(D_k^{\beta}(e^n), C_k^{\beta}(e^n)\big),
\end{equation}
and
\begin{equation}\label{stabA1}
\begin{split}
\big(\mathcal{G}[C_k^{\beta}(\phi^n)]-\mathcal{G}[C_k^{\beta}(\phi(t^n))],C_k^{\beta}(e^n)\big) &\le \mathop{\rm {min}}\limits_{\varepsilon>0}\big(\frac{\varepsilon}{2}\|\mathcal{G}[C_k^{\beta}(\phi^n)]-\mathcal{G}[C_k^{\beta}(\phi(t^n))]\|_{\star}^2+\frac{1}{2\varepsilon}\|C_k(e^n)\|^2\big)\\
& \le \mathop{\rm {min}}\limits_{\varepsilon>0}\big(\frac{\varepsilon\gamma}{2} \|C_k^{\beta}(e^n)\|^2+\frac{1}{2\varepsilon} \|C_k^{\beta}(e^n)\|^2\big)\\
& \overset{\varepsilon=\frac{1}{\sqrt{\gamma}}}=\sqrt{\gamma}\|C_k^{\beta}(e^n)\|^2.
\end{split}
\end{equation}
As a result, the stability condition \eqref{stabcond} is derived by requiring $\eta_k(\beta)>\sqrt{\gamma} $ since the term $\big(D_k^{\beta}(e^n), C_k^{\beta}(e^n)\big)$ can be handled by Lemma \ref{Gstability} and Theorem \ref{Thm1}.

On the other hand, for the classical IMEX$k$ $(k=2,3,4)$ schemes, i.e., \eqref{genIMEX} with $\beta=1$, the suitable multipliers are given as $e^n-\tilde{\eta}_k e^{n-1}$ \cite{nevanlinna1981} and the smallest possible values of $\tilde{\eta}_k$ are
\begin{equation}\label{eta_tilde}
\tilde{\eta}_2=0,\quad \tilde{\eta}_3=0.0836,\quad \tilde{\eta}_4=0.2878.
\end{equation}

Hence,  the corresponding versions of \eqref{stabA1} and \eqref{stabA2} become
\begin{equation}\label{stabB1}
\begin{split}
& \big(\mathcal{G}[\sum_{q=0}^{k-1}c_{k,q}(1)\phi^{n-k+1+q}]-\mathcal{G}[\sum_{q=0}^{k-1}c_{k,q}(1)\phi(t^{n-k+1+q})],  e^n-\tilde{\eta}_k e^{n-1} \big) \\& \le \mathop{\rm {min}}\limits_{\varepsilon>0}\big(\frac{\varepsilon}{2}\|\mathcal{G}[\sum_{q=0}^{k-1}c_{k,q}(1)\phi^{n-k+1+q}]-\mathcal{G}[\sum_{q=0}^{k-1}c_{k,q}(1)\phi(t^{n-k+1+q})]\|_{\star}^2+\frac{1}{2\varepsilon}\|e^n-\tilde{\eta}_k e^{n-1}\|^2\big)\\
& \le \mathop{\rm {min}}\limits_{\varepsilon>0}\big(\frac{\varepsilon \gamma}{2}\sum_{q=0}^{k-1}|c_{k,q}(1)|\|e^{n-k+1+q}\|^2+\frac{1}{2\varepsilon}\|e^n-\tilde{\eta}_ke^{n-1}\|^2  \big)\\
& \le \mathop{\rm {min}}\limits_{\varepsilon>0}\big(\frac{\varepsilon \gamma}{2}\sum_{q=0}^{k-1}|c_{k,q}(1)|\|e^{n-k+1+q}\|^2+\frac{1}{2\varepsilon}(\|e^n\|^2+\tilde{\eta}_k^2\|e^{n-1}\|^2) \big),
\end{split}
\end{equation}
where $c_{k,q}(1)$ are defined in \eqref{IMEX2coeff}-\eqref{IMEX4coeff} with $\beta=1$, and
\begin{equation}\label{stabB2}
\big(\mathcal{L}e^n, e^n-\tilde{\eta}_ke^{n-1} \big) = \|e^n\|^2-\tilde{\eta}_k\big(\mathcal{L}e^n, e^{n-1} \big) \ge \|e^n\|^2-\frac{\tilde{\eta}_k}{2}(\|e^n\|^2+\|e^{n-1}\|^2).
\end{equation}
Combining \eqref{stabB1} and \eqref{stabB2}, we obtain the following stability condition for the classical IMEX type scheme with multiplier $e^n-\tilde{\eta}_ke^{n-1}$,
\begin{equation}\label{stabcond2}
1-\tilde{\eta}_k>\mathop{\rm {min}}\limits_{\varepsilon>0}\big(\frac{\varepsilon \gamma}{2}\sum_{q=0}^{k-1}|c_{k,q}(1)|+\frac{1}{2\varepsilon}(1+\tilde{\eta}_k^2) \big)\ge \sqrt{\tilde{c}_k\gamma(1+\tilde{\eta}_k^2)},
\end{equation}
with $\tilde{c}_k=\sum_{q=0}^{k-1}|c_{k,q}(1)|$.
Comparing  \eqref{stabcond} with \eqref{stabcond2}, we  have two remarks:
\begin{itemize}
\item  From \eqref{eta_tilde} and \eqref{stabcond2}, we observe that for the classical  IMEX schemes, higher-order  (i.e., larger $k$) requires stronger stability condition on the parameter $\gamma$ appearing in \eqref{LocalLp}. It is this  requirement on the time step that  limits the use of high order scheme in practice.
\item On the other hand, for the new class of IMEX schemes,   we observe from \eqref{eta} and \eqref{stabcond} that   the stability condition on $\gamma$ becomes weaker as we increase  $\beta$. In particular, the new higher-order schemes  with a suitable $\beta$ can be stable {\color{black}with a larger time step than that is  allowed with a   classical IMEX scheme of the same-order. For example, we have from \eqref{eta} that $\eta_2(2)=\eta_3(3)=\eta_4(5)=1/2$ which indicates that the stability condition \eqref{stabcond} of the new fourth-order scheme  with $\beta=5$ and third-order scheme with $\beta=3$  is the same as  that of the second-order classical scheme. Our numerical results in Example 3 below indicate that we can use the maximum allowable time step of the second-order classical scheme   in our new third- and fourth-order schemes to obtain more accurate results.}
\end{itemize}
\begin{rem}
{\color{black} Note that a new multiplier $e^n-\frac{2}{169} e^{n-1}-\frac{11}{169} e^{n-2}$ for the classical BDF3 scheme is reported in \cite{akrivis2016backward} and since $\hat{\eta}_3:=\frac{2}{169}+\frac{11}{169} < \tilde{\eta}_3=0.0836$, one can obtain milder conditions on $\gamma$ compared to adopting the Nevanlinna-Odeh multipliers. Nevertheless, we can derive even milder conditions on $\gamma$ by choosing larger $\beta$ in our new methods.}
\end{rem}
\section{Extension to fifth-order}
In Theorem \ref{Thm1}, we found suitable multipliers for the second- and third-order scheme with $ \beta \ge 1$ and for the fourth-order scheme with $\beta \ge 2$. 
In this section, we would like to show numerically that the multiplier we found in section \ref{multiplier} also works for the fifth-order scheme.

Following the same notations as before, we can obtain the coefficients $a_{5,q}(\beta), b_{5,q}(\beta), c_{5,q}(\beta)$ by solving the linear systems \eqref{solve_akq}, \eqref{solve_bkq} and \eqref{solve_ckq} with $k=5$, respectively. Then we can define $A_5^{\beta}(\phi^i), B_5^{\beta}(\phi^i), C_5^{\beta}(\phi^i)$ as in \eqref{IMEX234coeff}. Next,  we split $B_5^{\beta}(\phi^{n+1})$ as
\begin{equation}\label{eta5}
B_5^{\beta}(\phi^{n+1})=\eta_5(\beta)C_5^{\beta}(\phi^{n+1})+D_5^{\beta}(\phi^{n+1}),\quad \rm{with} \quad \eta_5(\beta)=\frac{\beta-1}{\beta+15},
\end{equation}
and define $\tilde{A}_5^{\beta}(\zeta), \tilde{C}_5^{\beta}(\zeta), \tilde{D}_5^{\beta}(\zeta)$ as in \eqref{poly}. Following the key steps in the proof of Theorem \ref{Thm1}, we present a sequence of numerical results to show that  $C_5^{\beta}(\phi^{n+1})$ is a suitable multiplier for the fifth-order scheme with $6.5 \le \beta \le 100$.
\begin{itemize}
\item  We have $\rm{gcd}\big(\tilde{A}^{\beta}_5(\zeta),\zeta\tilde{C}^{\beta}_5(\zeta)\big)=\rm{gcd}\big(\tilde{D}^{\beta}_5(\zeta),\tilde{C}^{\beta}_5(\zeta)\big)=1$ since $\tilde{A}_5^{\beta}(0)=a_{5,0} \neq 0$ and
\begin{equation}
\begin{split}
& \rm{det}\,Sly(\tilde{A}^{\beta}_5,\tilde{C}^{\beta}_5) =\frac{\beta^{12}}{221184} + \frac{11\beta^{11}}{110592} + \frac{635\beta^{10}}{663552} + \frac{78937\beta^9}{14929920} + \frac{552809\beta^8}{29859840} + \frac{638383\beta^7}{14929920} \\
&+ \frac{9801769\beta^6}{149299200} + \frac{4912619\beta^5}{74649600} + \frac{765683\beta^4}{18662400} + \frac{225157\beta^3}{15552000} + \frac{6143\beta^2}{2488320} + \frac{2071\beta}{10368000} + \frac{1}{160000}>0,
\end{split}
\end{equation}
and
\begin{equation}
    \rm{det}\,Sly(\tilde{D}^{\beta}_5,\tilde{C}^{\beta}_5)=\frac{\beta^3(\beta^3 + 6\beta^2 + 11\beta + 6)^3}{13824}>0.
    \end{equation}
    \item  Let $r_1,r_2,...,r_5$ be the five roots of  $\tilde{C}^{\beta}_5(\zeta)=0$, and denote $r_{\rm max}=\mathop{\rm{max}}\limits_{1\le i \le5}|r_i|$. In Fig. \ref{fig:C5root}, we plot   the numerical values of $r_{\rm max}$ for $0 \le \beta \le 100$. We observe that   $r_{\rm max}<1$ for $0 \le \beta \le 100$, which implies $\tilde{C}^{\beta}_5(\zeta)$ is holomorphic outside the unit disk in the complex plane.
  \begin{figure}[!htbp]
 \centering
{ \includegraphics[scale=.3]{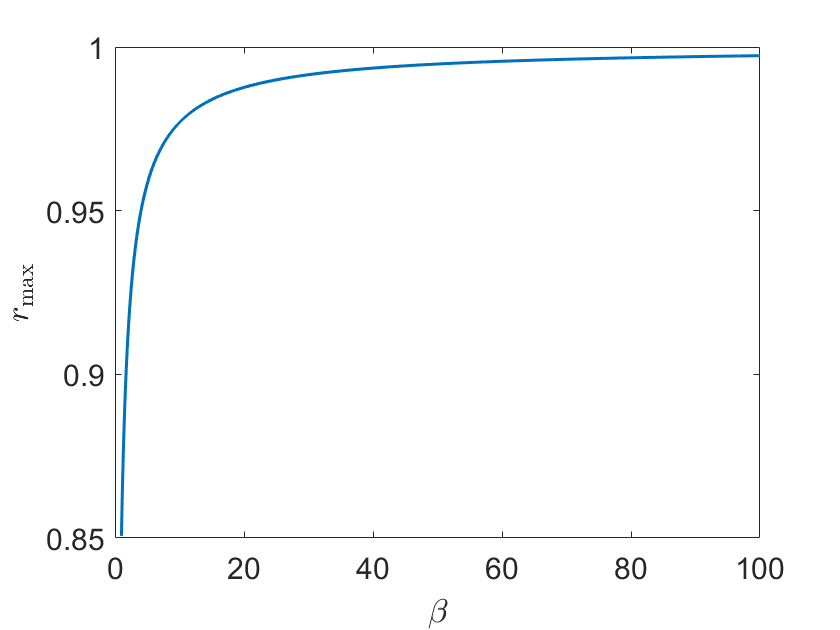}}
\caption{$r_{\rm max}$ with different $\beta$. }\label{fig:C5root}
 \end{figure}
 \item  Following the same process as in the proof of Theorem \ref{Thm1}, we can derive that $\text{Re}\frac{\tilde{A}_5^{\beta}(\zeta)}{\zeta\tilde{C}_5^{\beta}(\zeta)}>0\, \text{for}\, |\zeta|>1$ is equivalent to
     \begin{equation*}
     \frac{1}{180}(1-y)f_5(y)\ge 0,\quad \forall y \in[-1,1],
     \end{equation*}
     where
\begin{equation}\label{IMEXA5}
f_5(y)=\sigma_4(\beta)y^4+\sigma_3(\beta)y^3+\sigma_2(\beta)y^2+\sigma_1(\beta)y+\sigma_0 \ge 0,\quad \forall y \in[-1,1],
\end{equation}
with
\begin{subequations}
\begin{align}
& \sigma_4(\beta)=5\beta^8+ 70\beta^7+ 390\beta^6 + 1090\beta^5+ 1539\beta^4+ 820\beta^3- 350\beta^2 - 540\beta - 144,\\
& \sigma_3(\beta)=- 20\beta^8 - 280\beta^7 - 1550\beta^6 - 4260\beta^5  - 5836\beta^4- 3024\beta^3 + 950\beta^2+ 1396\beta+ 336,\\
& \sigma_2(\beta)=30\beta^8+ 420\beta^7+ 2310\beta^6+ 6240\beta^5+ 8244\beta^4  + 3932\beta^3 - 1260\beta^2 - 1340\beta- 204,\\
& \sigma_1(\beta)=- 20\beta^8 - 280\beta^7- 1530\beta^6- 4060\beta^5- 5136\beta^4 - 2072\beta^3+ 1070\beta^2+ 652\beta + 36,\\
& \sigma_0(\beta)=5\beta^8  + 70\beta^7   + 380\beta^6    + 990\beta^5  + 1189\beta^4  + 344\beta^3   - 410\beta^2   - 168\beta   + 336.
\end{align}
\end{subequations}

\item On the other hand, we can also show that  $\text{Re}\frac{\tilde{D}_5^{\beta}(\zeta)}{\tilde{C}_5^{\beta}(\zeta)}>0 \, \text{for}\, |\zeta|>1$ is equivalent to
\begin{equation}\label{IMEXD5}
h_5(y)=\mu_4(\beta)y^4+\mu_3(\beta)y^3+\mu_2(\beta)y^2+\mu_1(\beta)y+\mu_0(\beta) \ge 0, \quad \forall y \in [-1,1],
\end{equation}
with
\begin{small}
\begin{subequations}
\begin{align}
& \mu_4(\beta)=\frac{\beta(\beta^2 + 3\beta + 2)^2(6\beta^3 + 37\beta^2 + 48\beta -27)}{18(\beta + 15)},\\
& \mu_3(\beta)=-\frac{\beta(24\beta^7 + 292\beta^6 + 1366\beta^5 + 3013\beta^4 + 2881\beta^3 +193\beta^2 - 1391\beta - 618)}{18(\beta + 15)},\\
& \mu_2(\beta)=\frac{\beta(12\beta^7 + 146\beta^6 + 670\beta^5 + 1385\beta^4 + 1021\beta^3 - 553\beta^2 - 1127\beta - 402)}{6(\beta + 15)},\\
& \mu_1(\beta)=-\frac{(24\beta^8 + 292\beta^7 + 1314\beta^6 + 2527\beta^5 + 1203\beta^4 - 2405\beta^3 - 3117\beta^2 - 1008\beta - 270)}{18(\beta + 15)},\\
& \mu_0(\beta)=\frac{6\beta^8 + 73\beta^7 + 322\beta^6 + 571\beta^5 +91\beta^4 - 926\beta^3 - 995\beta^2 - 312\beta + 18}{18(\beta + 15)}.
\end{align}
\end{subequations}
\end{small}
In Fig. \ref{fig:fh5}, we plot the minimum values of $f_5(y)$ and $h_5(y)$ in $[-1,1]$ with $1\le \beta \le 100$, which show \eqref{IMEXA5} is true for $1\le \beta \le 100$ and \eqref{IMEXD5} is true for $6.5 \le \beta \le 100$.
Therefore, we have numerically verified that Theorem \ref{Thm1} is also true for  \eqref{genIMEX} with $k=5$ and $6.5 \le \beta \le 100$.

\begin{rem}
The choice of $\eta_5(\beta)$ in \eqref{eta5} is not unique, and the range $6.5 \le \beta \le 100$ is not necessarily the largest possible. But  our numerical results indicate \eqref{IMEXA5} and \eqref{IMEXD5} do not hold for some $\beta>100$.  

For the sixth-order scheme, our numerical results show there exists $|r_6|>1$, which is one root of $\tilde{C}^{\beta}_6(\zeta)=0$ and this implies that it is not holomporphic outside the unit disk. Hence, the proof in Theorem \ref{Thm1} can not be extended to the sixth-order.
\end{rem}

\begin{figure}[!htbp]
 \centering
  \subfigure{ \includegraphics[scale=.28]{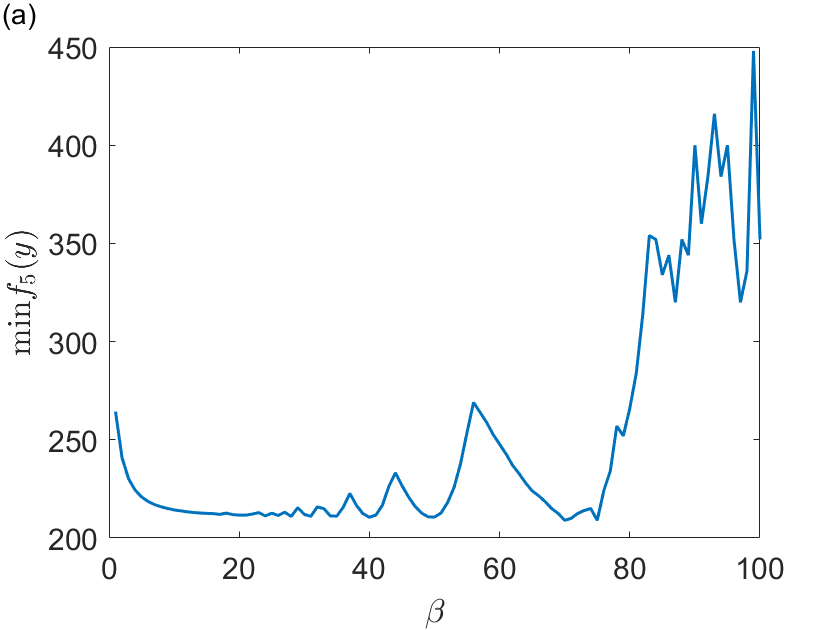}}
  \subfigure{ \includegraphics[scale=.28]{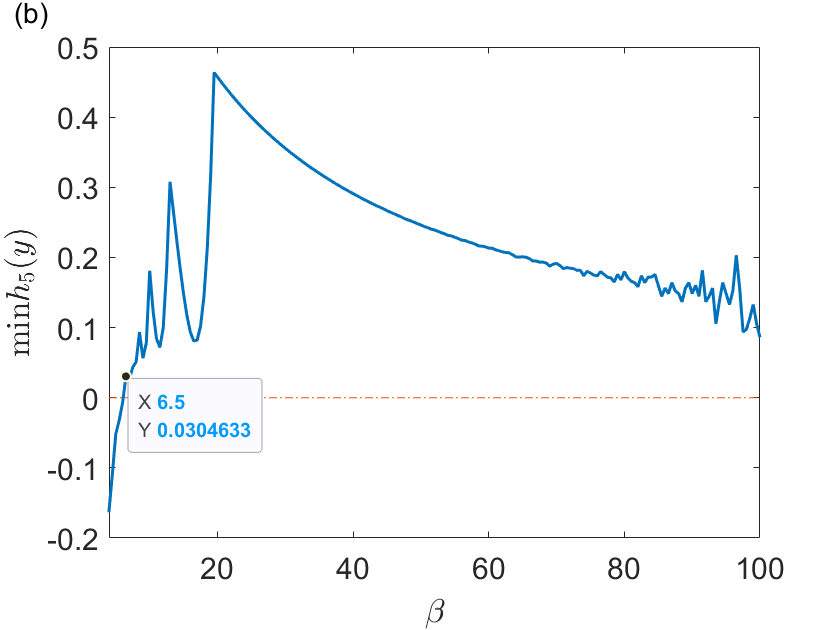}}
  \caption{Minimum value of $f_5$ and $h_5$ in $[-1,1]$ with different $\beta$. }\label{fig:fh5}
 \end{figure}
\end{itemize}

\section{Numerical examples}
In this section, we provide some numerical approximation of  the Allen-Cahn \cite{allen1979microscopic} and Cahn-Hilliard \cite{cahn1958free} equations to validate our theoretical results, and to show the advantages of the new IMEX  schemes \eqref{genIMEX}.

Given a free energy
\begin{equation}
\mathcal{E}[\phi]=\int\frac{1}{2}|\nabla \phi|^2+\frac{1}{4\varepsilon^2}(1-\phi^2)^2 d \bm x.
\end{equation}
We consider the $H^{-\alpha}$ gradient flow,
\begin{equation}\label{ac_ch}
\frac{\partial \phi}{\partial t}=-m(-\Delta)^{\alpha}\big(-\Delta \phi-\frac{1}{\varepsilon^2}\phi(1-\phi^2)\big)+f(t),\quad \alpha=0 \quad \rm{or} \quad 1,
\end{equation}
where $f$ is the given source term. When $\alpha=0$, \eqref{ac_ch} is the standard Allen-Cahn equation; when $\alpha=1$, it becomes the standard Cahn-Hilliard equation.

\textit{Example 1.} In the first example, we validate the convergence order of the new schemes. Considering a two-dimensional domain {\color{black} $(0, 2)^2$} with periodic boundary conditions, let $\alpha=0$, $m=\varepsilon=0.2$ in \eqref{ac_ch} and $f$ is chosen such that the exact solution of \eqref{ac_ch} is
\begin{equation}
\phi(x,y,t)=e^{\sin(\pi x)\sin(\pi y)}\sin(t).
\end{equation}
We use the Fourier Galerkin method with $Nx=Ny=40$ in space   so that the spatial discretization error is negligible compared to the time discretization error. In Fig. \ref{fig:convergence}, we plot the convergence rate of the $L^2$ error at $T=1$ by using the second- to fourth- order schemes \eqref{genIMEX}. We  observe the expected convergence order for all the cases with different  $\beta$. We also observe that for the same order, the error increases   slightly with  larger $\beta$.
\begin{figure}[!htbp]
 \centering
  \subfigure{ \includegraphics[scale=.35]{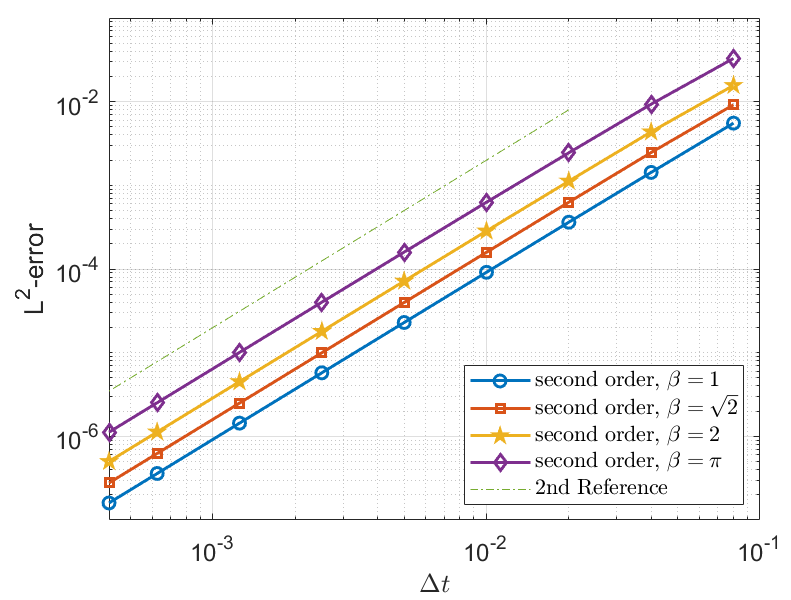}}
  \subfigure{ \includegraphics[scale=.35]{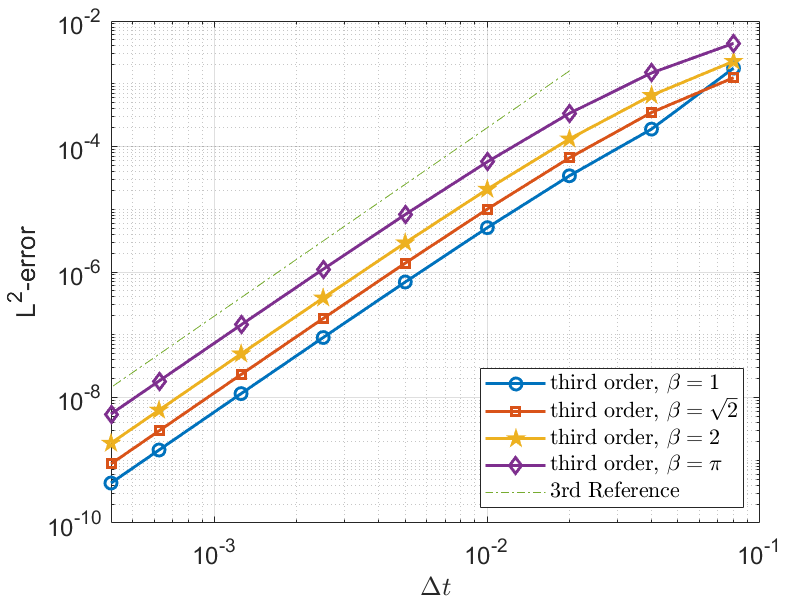}}
  \subfigure{ \includegraphics[scale=.35]{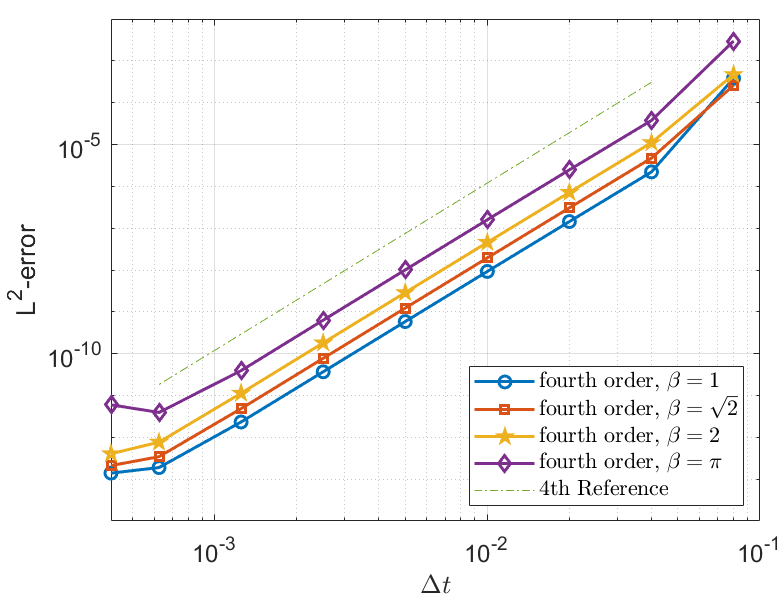}}
  \caption{ Convergence test for the general IMEX type methods. From left to right: second order, third order and fourth order schemes with different $\beta$. }\label{fig:convergence}
 \end{figure}

\textit{Example 2.} In the second example, we solve a benchmark problem for the Allen-Cahn equation \cite{chen1998applications}. Consider a two-dimensional domain $(-128,128)^2$  with a circle of radius $R_0=100$. In other words, the initial condition is given as
\begin{equation}
\phi(x,y,0)=\left\{
\begin{aligned}
&1,\quad x^2+y^2<100^2,\\
&-1,\quad x^2+y^2 \ge 100^2.
\end{aligned}
\right.
\end{equation}
By mapping the domain to $(-1,1)^2$, the parameters in \eqref{ac_ch} are given by $m=6.10351\times 10^{-5}$, $\varepsilon=0.0078$, $\alpha=0$ and $f=0$. In the sharp interface limit, the radius at time $t$ is given by
\begin{equation}
R=\sqrt{R_0^2-2t}.
\end{equation}
We use the Fourier Galerkin method with $Nx=Ny=512$ in space. Then we fix $\Delta t=0.75$, {which is the maximum time step we can use for the classical second-order scheme to get acceptable numerical results,} and use  \eqref{genIMEX} with different orders  and different $\beta$. We plot the computed radius $R(t)$ in Fig. \ref{fig:AC}, which shows that we can use higher-order schemes with the same large time step as the second-order schemes by choosing $\beta>1$. More importantly, we can get much more accurate results  with higher-order schemes. Here, $k=1, \beta=1$ represents the usual first-order scheme.


\begin{figure}[!htbp]
 \centering
  \subfigure{ \includegraphics[scale=.55]{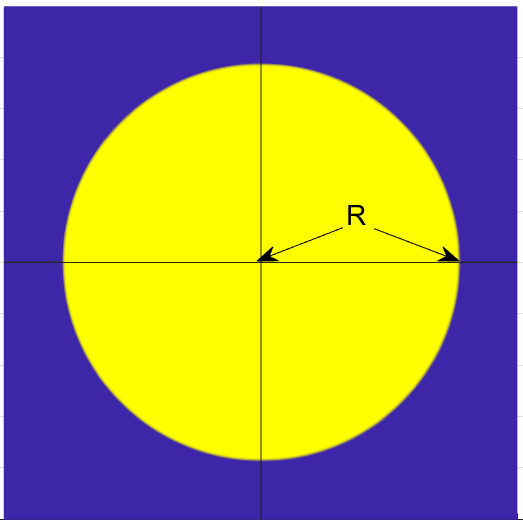}}\qquad\qquad
  \subfigure{ \includegraphics[scale=.48]{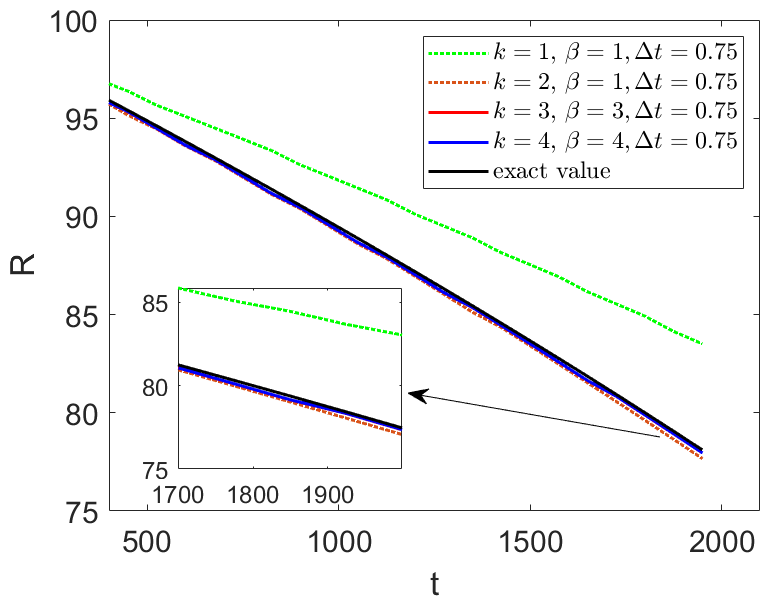}}\\
  \caption{ The evolution of radius $R$ with $\Delta t=0.75$ under different schemes.   }\label{fig:AC}
 \end{figure}

\textit{Example 3.} In the third example, we consider the Cahn-Hilliard equation in a two-dimensional domain $(0,1)^2$ with periodic boundary condition and let $\alpha=1$, $m=1$, $\varepsilon=0.02$ in \eqref{ac_ch}. The initial condition is given as $\phi(0)=0.2+r$ and $r$ is a random perturbation variable with uniform distribution in $[-0.02,0.02]$. We use the Fourier Galerkin method with $Nx=Ny=128$ in space. In Fig. \ref{fig:CH}, we compare the first- to the fourth-order schemes with different $\beta$, the reference solution is generated by using the classical fourth-order scheme with sufficiently small time step $\Delta t= 5\times 10^{-9}$.

Several observations are in order:
\begin{itemize}
\item 1. We take {\color{black}$\Delta t=7.5 \times 10^{-8}$} which is the maximum allowable time step for the classical second-order scheme, and  observe in Fig. \ref{fig:CH}(a) that we can use the same time step  for the higher-order schemes by choosing a suitable  $\beta>1$, and  obtain more accurate results.
\item 2. We observe in Fig. \ref{fig:CH}(b)  that the usual third- and fourth-order schemes with $\beta=1$ are unstable, but we can get correct solutions with the third- and fourth-order schemes by choosing a suitable  $\beta>1$.
\item 3. We also  observe in Fig. \ref{fig:CH}(b) that  $\beta$  too large may lead to inaccurate results due to larger truncation errors.
\end{itemize}
\begin{figure}[!htbp]
 \centering
  \subfigure[Comparisons of different order schemes]{ \includegraphics[scale=.32]{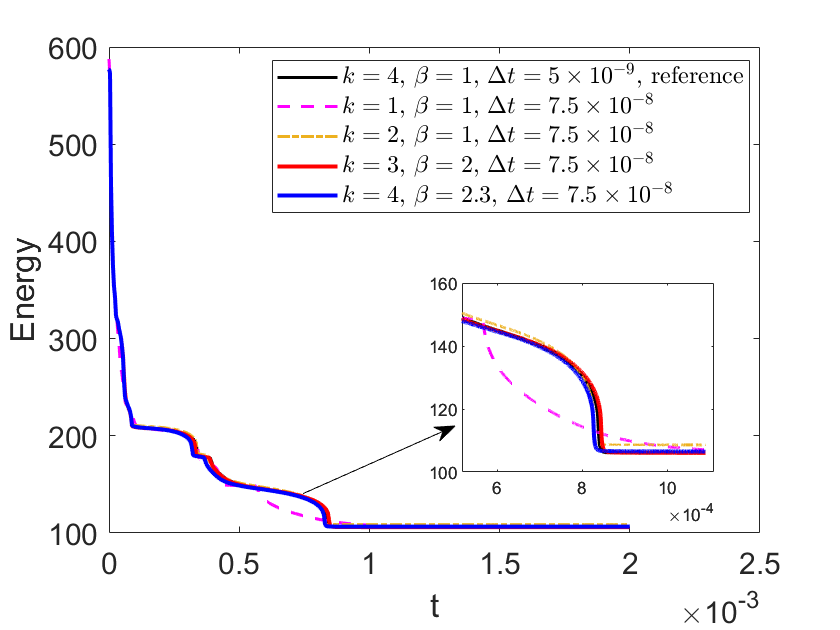}}
  \subfigure[Effect of $\beta$ in high order schemes]{ \includegraphics[scale=.32]{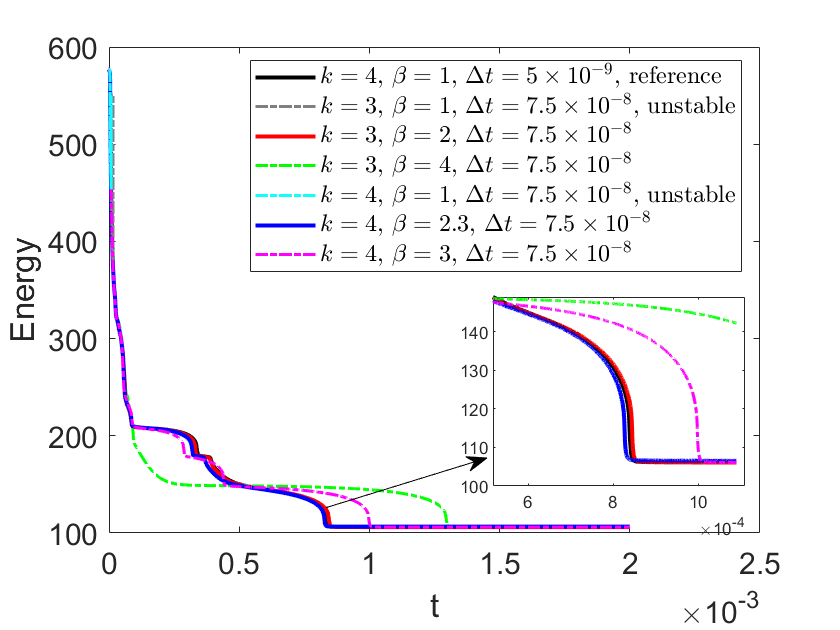}}\\
  \caption{ Comparisons of different order schemes with different $\beta$ for the Cahn-Hilliard equation}\label{fig:CH}
 \end{figure}
\section{Concluding remarks}

We presented in this paper a new class of BDF and IMEX  schemes  for  parabolic type equations based on the Taylor expansion at time $t^{n+\beta}$ with $\beta > 1$ being a tunable parameter.  The new schemes are a simple generalization of the classical  BDF or IMEX schemes with essentially the same computational efforts. However, they enjoy a remarkable property that their stability regions increase as the parameter $\beta$ increases, making it possible, by choosing a suitably large $\beta$,  to use high-order schemes with larger time steps that are only allowed with lower-order classical schemes.
We also  identified an explicit uniform  multiplier for the new  schemes of second- to fourth-order, and  carried out a  rigorous stability and error analysis   by using the energy argument. We also presented  numerical examples to show the benefit of using higher-order schemes with a suitable $\beta>1$.

This class of  new BDF and IMEX  schemes makes it possible to use higher-order schemes for highly stiff systems with reasonably large time steps, and can be easily  implemented with a minimal effort by modifying  the  code based on the classical BDF or IMEX schemes. Thus, it provides a much needed improvement on the stability of higher-order schemes.
The  idea behind the new class of BDF and IMEX schemes is very simple but original,  and can be extended to other type of numerical schemes.

\bibliographystyle{plain}
\bibliography{New_BDF}
\end{document}